\documentclass[11pt]{amsart}
\usepackage{latexsym,enumerate}
\usepackage{epsfig}       
\usepackage{subfigure}    
\usepackage{graphicx}
\usepackage{float}
\usepackage{multirow}
\usepackage{array}
\usepackage{color}
\usepackage{amsmath,amsfonts,amssymb,mathrsfs}
\usepackage{amssymb,mathrsfs,graphicx}
\usepackage{mathbbol}
\usepackage{caption}
\usepackage{hyperref}
\usepackage{algorithm}
\usepackage{algpseudocode}
\usepackage{bbding}
\usepackage[mathlines]{lineno}

\usepackage{xcolor}

\usepackage{ifthen}
\usepackage{hyperref}
\usepackage{enumerate}
\usepackage{algorithm}
\usepackage{epstopdf}
\usepackage{tikz}
\usepackage{pgfplotstable}
\usepackage{pifont}
\newcommand{\argmin}{\arg\!\min}
\newcommand{\argmax}{\arg\!\max}
\allowdisplaybreaks

\title[Oscillatory associative-memory network]
{Enhanced error-free retrieval in Kuramoto-type associative-memory networks via two-memory configuration} 

\author[Z. Li]{Zhuchun Li}
\address[Zhuchun Li]{\newline School of Mathematics and Institute for Advanced Study in Mathematics \newline Harbin Institute
	of Technology, Harbin 150001, China}
\email{lizhuchun@hit.edu.cn}

\author[X. Zhao]{Xiaoxue Zhao}
\address[Xiaoxue Zhao]{\newline School of Mathematics \newline Harbin Institute of Technology, Harbin 150001, China\\ \newline Department of Data Science \newline City University of Hong Kong, Kowloon 999077, Hong Kong Special Administrative Region}
\email{zhaoxiaoxue@hit.edu.cn, xiaoxue.zhao@cityu.edu.hk}

\author[X. Zhou]{Xiang Zhou}
\address[Xiang Zhou]{\newline Department of Mathematics \newline City University of Hong Kong, Kowloon 999077, Hong Kong Special Administrative Region}
\email{xizhou@cityu.edu.hk}

\newtheorem{theorem}{Theorem}[section]
\newtheorem{lemma}{Lemma}[section]

\newtheorem{proposition}{Proposition}[section]
\newtheorem{remark}{Remark}[section]
\newtheorem{example}{Example}[section]

\def\charf {\mbox{{\text 1}\kern-.30em {\text l}}}

\begin{document}
	
	\date{\today}
	
	
	
	
	\begin{abstract}
		We study the associative-memory network of Kuramoto-type oscillators that stores a set of memorized patterns (memories). In [\textit{Phys. Rev. Lett.}, 92 (2004), 108101], Nishikawa, Lai and Hoppensteadt showed that the capacity of this system for pattern retrieval with small errors can be made as high as that of the Hopfield network. Some stability analysis efforts focus on mutually orthogonal memories; however, the theoretical results do not ensure error-free retrieval in general situations. In this paper, we present a route for using the model in pattern retrieval problems with small or large errors. We employ the eigenspectrum analysis of Jacobians and potential analysis of the gradient flow to derive the stability/instability of binary patterns. For two memories, the eigenspectrum of Jacobian at each pattern can be specified, which enables us to give the critical value of the parameter to distinguish the memories from all other patterns in stability. This setting of two memories substantially reduces the number of stable patterns and enlarges their basins, allowing us to recover defective patterns. We extend this approach to general cases and present a deterministic method for  ensuring error-free retrieval across a general set of standard patterns.  Numerical simulations and comparative analyses illustrate the approach.

	\end{abstract}
	\maketitle \centerline{\date}
	
	\section{Introduction}
	\setcounter{equation}{0}
	
The human brain has been called the most complex object in the universe,  and researchers have made efforts to understand how it stores memories.   Physicists hope to learn   about memory by studying brain like   neural networks.  The associative-memory  network provides a basic idea for  neural computing \cite{AT,AT2,F-M-R-P, H-K, Ho-K1,Ho-K,H}, in which the neurons are assumed to be discrete values (e.g., +1 and 1)  and a set of patterns is stored in the network such that when a new pattern is presented, the network responds by producing a stored pattern that most closely resembles the new pattern.   Here, the stored patterns or memories are reflected by  dynamically {stable attractors}.

  In \cite{N-L-H}, an associative-memory network was developed  and it is shown that it can retrieve stable ``memories''. The model reads
%
\begin{linenomath*}
	\begin{equation}\label{mod1}
	{\dot \varphi}_i  =\frac{1}{N}\sum_{j=1}^{N}
	C_{ij}\sin(\varphi_{j}  - \varphi_i )+\frac{\varepsilon}{N}\sum_{j=1}^{N}\sin2(\varphi_{j} -\varphi_{i} )
	\end{equation}
	\end{linenomath*}
	where $i\in [N] :=\{1,2,\dots,N\}$ and 
each periodic oscillator (neuron) is dynamically described by a phase variable $\varphi_i(t)$. The network, reflecting the so-called Hebbian learning rule \cite{H2}, is set as $C_{ij}:=\sum_{k=1}^{M}\xi^k_i\xi^k_j$, where $\{\xi^1, \xi^2,\dots, \xi^M\}\subset\{1,-1\}^N$ is a set of so-called  memorized   patterns (memories) for the network.  When $C_{ij}\equiv K$ is a positive constant and $\varepsilon =0$, the system \eqref{mod1} falls into the celebrated Kuramoto model, and many results on the classic or variant Kuramoto   models can be found in   literature \cite{C-H-J-K,C-S,D-B1,D-X,H-N-P}.
\vspace{0.1cm}

   Notice that the binary patterns  $\xi^k$ and $-\xi^k$ can be regarded as   the same pattern.  In fact, if the   pattern  $\xi^k$ is replaced by $-\xi^k$,  the coupling term $C_{ij}$ does not change.
	The coefficient $\varepsilon>0$ is the strength of the second-order Fourier term, which can be regarded as an adjustable parameter that controls the stability/instability  of equilibria and leads to rich dynamical properties.
		For any binary pattern $\eta\in\{1,-1\}^N$, there is a unique (up to constant translation) phase-locked  bipolar state $\varphi^*(\eta)$ corresponding to the pattern $\eta$, which is characterized by $|\varphi_{i}^*(\eta)-\varphi_{j}^*(\eta)|=0 \;\mathrm{mod}\; 2\pi$ if $\eta_{i}=\eta_{j}$, and $|\varphi_{i}^*(\eta)-\varphi_{j}^*(\eta)|=\pi\; \mathrm{mod} \;2\pi$ if $\eta_{i}\ne\eta_{j}$.  The overlap
	\begin{align}\label{overlap}
	 {m(\eta)=\Big|\frac{1}{N} \sum_{i=1}^{N}\eta_{i} e^{\sqrt{-1}\varphi_i(t)}\Big|}
	\end{align}
	measures the closeness of the transient state $\varphi(t)$ to the pattern $\eta$. Due to the global phase shift invariance in  \eqref{mod1},  $m(\eta)$ is invariant under global rotations.
\vspace{0.1cm}

  The  work   \cite{N-L-H} was featured in Physical Review Focus \cite{trick}, as ``A Trick to Remember'' with a computer model for   how the brain stores memories.    If we  assign  values to each oscillator   and then let the system go on its own, the phases change with time and eventually   settle  down to an unchanging pattern of values.  The setup of $C_{ij}$ creates a network with several stable  patterns of neuron values. When we start the network in some other pattern, it is expected to evolve  towards the     closest memory.
Nishikawa,  Lai and  Hoppensteadt showed that  the  capacity of the   associative-memory networks can be
made as high as that of the Hopfield network in probability as $N$ being large.  
Here,   the capacity is  defined by the transition point at which the states that encode {memorized} patterns with small amount of error become unstable (or cease to exist) \cite{N-H-L, N-L-H}.  Therefore, it is   desired  to  make the memories stable and other patterns with errors   unstable.  However, in general the system is  multi-stable \cite{G-H-M, P-H}, and    more patterns become stable if we setup a larger $\varepsilon$.  When there are many (more than three) memories,   many  patterns other than   memories can  become stable once we setup $\varepsilon$  making memories stable. This can cause mistakes when we apply the model for pattern retrieval. In particular,   error-free  retrieval fails,  which requires that the  retrieved pattern must coincide with one of the memories.    
\vspace{0.1cm}


 \textbf{Motivation.}	
In this paper, we explore a deterministic  approach for  application of model \eqref{mod1} to the following binary pattern retrieval problems:
	let $\{\eta^1, \eta^2,\dots, \eta^M\}$ be a given set of standard binary patterns, and $\bar\eta$ be a defective copy of a standard pattern  with small or large errors,  our aim is to recognize the standard pattern closest to the defective one.
	For this, a principal problem is the stability/instability of binary patterns. In fact, it is  highly desired that the standard patterns are stable, but those    patterns other than the standard ones  become unstable  and then cease to be retrieved.    As a typical and simplified manner \cite{N-H-L, N-L-H}, one can  setup the network \eqref{mod1} by memorizing the whole set of standard patterns. 

	\vspace{0.1cm} 

 There have been  some efforts   to find   conditions for the stability of memories and instability of other patterns.  In \cite{N-H-L, N-L-H},   the  stability/instability is conditional to  the eigenspectrum of Jacobian, but it can be not easy to determine since it requires calculations of eigenvalues of the Jacobian at each pattern. 
 In \cite{L-Z-X, Z-L-X},  the mutual orthogonality in memorized patterns was
incorporated  which guarantees that the memories are $\varepsilon$-independently stable, i.e., stable for all $\varepsilon>0$. For patterns without this stability, there is a positive critical strength $\varepsilon^*_\eta$ such that $\eta$ is  stable   if $\varepsilon>\varepsilon^*_\eta$ and unstable if $\varepsilon<\varepsilon^*_\eta$. 
 However,
  when there are more than three memories, some additional patterns other than the memories become  $\varepsilon$-independently stable, see \cite[Subsection 4.1]{Z-L-X}.   Therefore, it is impossible to distinguish the memories from these patterns   in stability. In particular, \cite[Theorem 4.3]{Z-L-X} shows that   more  patterns become $\varepsilon$-independently stable   if there are more memories. 
  Of course, simulations also  suggest that it is not a proper way to memorize three binary patterns, see \cite[Subsection 3.2]{L-Z-X}.
 Another drawback in the above theory is that the mutual orthogonality is unreal in  practical situations. Therefore, for the purpose of error-free retrieval in practical applications, two questions arise: 
\begin{itemize} \item[Q1.]  How to make the memories stable and
	other patterns with errors unstable?
	\item[Q2.]   How to find the theory   valid for the case without  orthogonality ? \end{itemize}

  To answer question Q1, we should reduce the multi-stability of the system, i.e., reduce the number of stable patterns. Inspired by the above theory in \cite{Z-L-X} for the special case with mutual orthogonality,  a possible solution is to   consider the case  with a few memories, which reduces the number of stable patterns.  This also  enlarges their basins so that  the defective patterns with large errors can be recovered.  In \cite{Z-L-X2}, a set of nonorthogonal standard binary patterns is grouped, with each subgroup containing at most two patterns. Then the subgroups with two patterns are orthogonalized and used as the mutually orthogonal memorized patterns of the system \eqref{mod1}. Although the orthogonalization addresses Q1 by ensuring the stability of the memories while the remaining binary patterns are unstable, the doubling of the system's dimensionality results in a longer operational time. In this paper, we investigate both Q1 and Q2 without  orthogonality.

	\vspace{0.1cm}
	
 \textbf{Contribution.}	  The main contribution of  this paper is  to  discover the  advantages  in this system    with  two memories and extend it to general cases. This then leads to a route for practical use of the model \eqref{mod1} in general pattern retrieval problems with error-free retrieval.    
Without the assumption of mutual orthogonality in two memories,  we clearly specify the eigenspectrum and corresponding eigenspaces of the Jacobian at each binary pattern, which enables us to choose appropriate $\varepsilon$ so  that the   memories are stable while all other patterns are unstable (Section \ref{sec2}), i.e., give an answer  for question Q1. 
Meanwhile, this naturally gives the answer for question Q2, to  overcome the deficiency in the work  \cite{Z-L-X,Z-L-X2} which requires mutual orthogonality in memories.   Although the theory is not valid  for the case with more memories (see  Section \ref{sec3} for the case of three memories), we can fortunately extend the approach to    problems with a general set of standard patterns (Section \ref{sec4}).   Therefore, we present a deterministic    approach    for general problems of  binary  pattern retrieval.   
  Simulations for defective patterns with large errors are presented as illustration.
  	
	\vspace{0.1cm}
	
	 {It is worth noting that we give the   critical range of $\varepsilon$'s to distinguish  the two memories from other patterns in stability, which is  important for practical use.    In Theorem \ref{onlytwo} and Proposition \ref{I1}, we figure out the critical strength of $\varepsilon$ below which    the two memories  are stabilized while all other binary patterns are destabilized.    
We see that the critical strength depends only on the two memories and is  easy to compute. } 
	\vspace{0.1cm}

  Another    problem in this model is concerning the basin of  stable equilibria. This question has attracted numerous  interests in mathematical physics and engineering, for example, \cite{C-H-J-K, D-B, W-P-P}.  The  common feature in those studies is that the couplings are non-repulsive and the  focus is on the basin of  synchronization. However, in the present model \eqref{mod1},  the Hebbian couplings between oscillators can be  attractive $(C_{ij}>0)$ or repulsive $(C_{ij}<0)$.   As far as we know, the basin estimate  has never been touched in this model, even for the sepcial case with mutually orthogonal memories. In this paper,   the  setting  of two memories  enables us to present a  nontrivial  estimate for the basin of stable patterns. Of course, this rigorous estimate  is conservative, but it gives some insight to understand how the dynamical approach works in pattern retrieval.
  \vspace{0.1cm}
	
	

	\textbf{Organization of paper.}  In Section \ref{sec2}, we consider the  system with two memories. We explore the critical strength for stability of memories and instability of other patterns, and estimate the basin of memories.  In Section \ref{sec3}, it is shown  that  those nice properties for two memories cannot be extended to   three memories.   In Section \ref{sec4}, the approach for solving general pattern retrieval problems are presented with some illustrative simulations. 
	Finally, a brief conclusion  is given in Section \ref{seccon}.

	\section{The system with two memories}\label{sec2}\setcounter{equation}{0}
In this section, we   study  the eigenspectrum of the Jacobian at each binary pattern to  explore the stability/{instability}   when {the system is} made up of two memories, which enables us to distinguish the memories from other patterns in stability.  

	We consider the  system \eqref{mod1} with  two memories $\{\xi^1,\xi^2\}\subset \{1,-1\}^N$, that is,
	\begin{linenomath*}
			\begin{equation}\label{mod}
		{\dot \varphi}_i =\frac{1}{N}\sum_{j=1}^{N}
		C_{ij}\sin(\varphi_{j}  - \varphi_i )+\frac{\varepsilon}{N}\sum_{j=1}^{N}\sin2(\varphi_{j}  -\varphi_{i}),\quad C_{ij}:= \xi^1_i\xi^1_j+\xi^2_i\xi^2_j,
		\end{equation}
	\end{linenomath*}
Throughout this section, we assume $\xi^1\neq \pm\xi^2$ since $\xi^k$ and $-\xi^k$  should be regarded as the same pattern.

 {Let $\varphi=(\varphi_1,\varphi_2,\dots,\varphi_N)$ and
\begin{align}\label{potential}
f_{\varepsilon}(\varphi)=-\frac{1}{2N}\sum_{i=1}^{N}\sum_{j=1}^{N}(\xi^1_i\xi^1_j+\xi^2_i\xi^2_j)\cos(\varphi_j-\varphi_i)-\frac{\varepsilon}{4N}\sum_{i=1}^{N}\sum_{j=1}^{N}\cos 2(\varphi_j-\varphi_i);
\end{align}
then  \eqref{mod} can be written as a gradient system with potential $f_{\varepsilon}$, i.e.,
\begin{align}\label{gradient}
\dot{\varphi}=-\nabla f_{\varepsilon}(\varphi).
\end{align}
 It was shown in   \cite{A-K} that,  for a gradient system with analytic potential, $\varphi^*$ is an (asymptotically) stable equilibrium of \eqref{gradient} if and only if $\varphi^*$ is a (strict) local minimum of $f_{\varepsilon}$.} As a consequence of \cite{LL-XX}, any solution of \eqref{mod} converges to some equilibrium.

For a  binary pattern $\eta\in \{1,-1\}^N$,  the Jacobian   of system \eqref{mod} at $\varphi^*(\eta)$ is
		\begin{linenomath*}
	\[\mathcal J_{\eta}=\begin{pmatrix}
	-T_{1} & \frac{1}{N}C_{12}\eta_{1}\eta_{2}+\frac{2\varepsilon}{N} & \frac{1}{N}C_{13}\eta_{1}\eta_{3}+\frac{2\varepsilon}{N} & \dots &\frac{1}{N}C_{1N}\eta_{1}\eta_{N}+\frac{2\varepsilon}{N}  \\
	\frac{1}{N}C_{21}\eta_{2}\eta_{1}+\frac{2\varepsilon}{N} & -T_{2} & \frac{1}{N}C_{23}\eta_{2}\eta_{3}+\frac{2\varepsilon}{N} & \dots & \frac{1}{N}C_{2N}\eta_{2}\eta_{N}+\frac{2\varepsilon}{N} \\
	\dots & \dots & \dots & \ddots & \dots \\
	\frac{1}{N}C_{N1}\eta_{N}\eta_{1}+\frac{2\varepsilon}{N} & \frac{1}{N}C_{N2}\eta_{N}\eta_{2}+\frac{2\varepsilon}{N} & \frac{1}{N}C_{N3}\eta_{N}\eta_{3}+\frac{2\varepsilon}{N} & \dots & -T_{N}
	\end{pmatrix},\]
\end{linenomath*}
	where
	$T_{i}=\frac{1}{N}\sum_{j=1,j\ne i}^{N}C_{ij}\eta_{i}\eta_{j}+\frac{2\varepsilon (N-1)}{N}$. This Jacobian $\mathcal J_\eta$ is dependent on the Hebbian coupling matrix  $\left\{C_{ij}\right\}$ and the pattern $\eta$ under consideration.

In order to clarify  the eigenspectrum of Jacobians,   we introduce some notations for the statement of eigenspaces.
 Let $I\subset[N]$, we define
	\[V(I)=\big\{ (x_1,x_2,\dots,x_N)^{\mathrm{T} }\in\mathbb{R}^N\,\big|\, x_i=0\;  \text{for} \;  i\notin I, \,\, \text{and} \,\, \textstyle{\sum_{i\in I}x_i=0}\big\}.\]  Let $I$ and $\hat{I}$ be two disjoint   subsets of $[N]$, we set
	\[V(I;\hat{I})= \{kx^\star(I;\hat{I})\in \mathbb R^N  \;\big|\;  k\in\mathbb{R}  \}\]
where  {$x^\star(I;\hat{I})=(x^\star_1,x^\star_2,\dots,x^\star_N)^{\mathrm{T} }$} is defined by  $$x_i^\star=\begin{cases}|\hat{I}|, &i\in I,\\ -|I|, &i\in \hat{I},\\ 0, &\text{otherwise}. \end{cases}$$
In particular, if one of $I$ and $\hat{I}$ is  empty, we set  $V(I;\hat{I})=\{\mathbb 0\}$, with $ \mathbb{0}=(0,0,\dots,0)^{\mathrm{T}}$.
 It should be noted that if $I\cup \hat{I}=[N]$, then the vector $x^\star(I;\hat{I})$ has no zero element.
	We also denote
	$V[\mathbb 1]=\{k\mathbb{1}\mid   k\in\mathbb{R}\}$ with $ \mathbb{1}=(1,1,\dots,1)^{\mathrm{T}} .$ It is easy to see that  {for $I, \hat{I}\neq\emptyset$,}
	\[ \dim V(I)=|I|-1,\quad  \dim V(I;\hat{I})=1,\quad \text{and}\quad \dim V[\mathbb 1]=1. \]
Let $V_1$ and $V_2$ be linear subspaces in $\mathbb{R}^N$, we set
	$V_1+V_2=\{y+z\in \mathbb{R}^N \,\big|\,   y\in V_1, z\in V_2\},$
	which is again a subspace.

\subsection{Stability/instability of binary patterns}\label{subsec1}
	
Consider the memories $\xi^1$ and $\xi^2$ in system \eqref{mod}. 	Let $\mathcal J_{\xi^1}$ and $ \mathcal J_{\xi^2}$ be the Jacobian  at $\varphi^*(\xi^1)$ and $\varphi^*(\xi^2)$, respectively.  Note that
\begin{linenomath*}
	 \[C_{ij}\xi^1_i\xi^1_j=(\xi^1_i\xi^1_j+\xi^2_i\xi^2_j)\xi^1_i\xi^1_j=(\xi^1_i\xi^1_j+\xi^2_i\xi^2_j)\xi^2_i\xi^2_j=C_{ij}\xi^2_i\xi^2_j ,\;\quad    i,j\in [N],\]
	\end{linenomath*}
	we have   $\mathcal J_{\xi^1}=\mathcal J_{\xi^2}$, which implies that $\varphi^*(\xi^1)$ and $\varphi^*(\xi^2)$ have  the same linear stability.
	In order to explore the spectrum, we introduce the following notation:
\begin{linenomath*}
	\begin{equation}\label{index}I_1=\left\{i\in  [N]\,\mid\, \xi^1_i=\xi^2_i\right\}\,\,\, \text{and} \,\,\, I_2=\left\{i\in [N]\,\mid\, \xi^1_i\neq\xi^2_i\right\}.\end{equation}
\end{linenomath*}
	By $\xi^1\neq \pm\xi^2$, we have $I_1\neq \emptyset$, $I_2\neq \emptyset$ and
	$I_1\cup I_2=[N]$, $I_1 \cap I_2=\emptyset$. 

	\begin{proposition}\label{xi12}
		Let $\{\xi^1,\xi^2\}$ be the set of memories in system \eqref{mod}. 	Then for $J_{\xi^1}$ (and $J_{\xi^2}$), we have:\\
  $(1)$ $0$ is a single eigenvalue with eigenspace $V[\mathbb 1]$;\\
		$(2)$ $-2\Big(\frac{|I_1|}{N}+\varepsilon\Big)$ is an eigenvalue  with eigenspace $V(I_1)$;\\
        $(3)$ $-2\Big(\frac{|I_2|}{N}+\varepsilon\Big)$ is an eigenvalue  with eigenspace $V(I_2)$;\\
    	$(4)$ $-2\varepsilon$ is a single eigenvalue with eigenspace $V(I_1;I_2)$.
	\end{proposition}
	\begin{proof}  {Note that the sum of dimensions of the four subspaces, $V[\mathbb 1], V(I_1), V(I_2)$ and $V(I_1;I_2)$, is $N$.} Therefore, it suffices to  show that any nonzero vector in each space   is an eigenvector of the corresponding  eigenvalue. 

	$(1)$ It is easy to check that  $\mathcal J_{\xi^1}\mathbb{1}=0\mathbb{1}$, since the elements in each row of $\mathcal J_{\xi^1}$ sum  to 0. In fact, this reflects the global phase shift invariance.

 {In order to prove (2)-(4), we first note that}  {for any    $x=(x_1,x_2,\dots,x_N)^{\mathrm{T}}\in V(I_1)\cup V(I_2)\cup V(I_1;I_2)$,  we have  $\sum_{\substack{j=1}}^{N}x_j=0$ and 
		\begin{align}	\begin{aligned} \label{a}
			 (\mathcal J_{\xi^1}x)_i&=\sum_{\substack{j=1,j\neq i}}^{N}\Big(\frac{1}{N}C_{ij}\xi^1_i\xi^1_j+\frac{2\varepsilon}{N}\Big)x_j+
\Big[-\frac{1}{N}\sum_{\substack{j=1,j\neq i}}^{N} C_{ij}\xi^1_i\xi^1_j-\frac{2\varepsilon(N-1)}{N}\Big]x_i\\
			&=\frac{1}{N}\Big[ \sum_{j=1}^{N}C_{ij}\xi^1_i\xi^1_jx_j- \sum_{j=1}^{N}C_{ij}\xi^1_i\xi^1_jx_i\Big]+\frac{2\varepsilon}{N}\sum_{\substack{j=1,j\neq i}}^{N}x_j-\frac{2\varepsilon(N-1)}{N}x_i\\
			&=\frac{1}{N} \sum_{j=1}^{N}C_{ij}\xi^1_i\xi^1_j(x_j- x_i) - 2\varepsilon x_i,
			\end{aligned}\end{align}
where $(\mathcal J_{\xi^1}x)_i$  represents the $i$-th component of   $\mathcal J_{\xi^1}x$.
Using $\xi^1_j\xi^2_j=1$ for $j\in I_1$, and $\xi^1_j\xi^2_j=-1$ for $j\in I_2$,  we have
	\begin{linenomath*}
		\begin{align}\label{cxx}
		C_{i j}\xi^1_{i}\xi^1_j=(\xi^1_{i}\xi^1_j+\xi^2_{i}\xi^2_j)\xi^1_{i}\xi^1_j=
1+\xi^1_{i}\xi^1_j\xi^2_{i}\xi^2_j=\begin{cases}1+\xi^1_{i}\xi^2_i , &j\in I_1,\\ 1-\xi^1_{i}\xi^2_i, &j\in I_2. \end{cases}
		\end{align}
	\end{linenomath*}  Next  we show (2)-(4).}

		$(2)$ For any $x=(x_1,x_2,\dots,x_N)^{\mathrm{T}}\in V(I_1)$, we will  show
		$\mathcal J_{\xi^1}x=-2\big(\frac{|I_1|}{N}+\varepsilon\big)x,$
		which  is equivalent to
		\begin{align}\label{equz}
		(\mathcal J_{\xi^1}x)_i=-2\big(\frac{|I_1|}{N}+\varepsilon\big)x_i\;\; \text{for}\;\; i\in I_1\quad \text{and}\quad (\mathcal J_{\xi^1}x)_i=0\;\; \text{for}\;\; i\in I_2.
		\end{align}
	Note that for  $x\in V(I_1)$, we have
$\sum_{j\in I_1}x_j=0$ and $x_j=0$ for $j\in I_2$. By \eqref{a}, \eqref{cxx} and $|I_1|+|I_2|=N$, we have
	\begin{linenomath*}
		\begin{align*}
	(\mathcal J_{\xi^1}x)_i& =\frac{1}{N}\Big[ \sum_{j\in I_1}C_{ij}\xi^1_i\xi^1_j(x_j-x_i)+\sum_{j\in I_2}C_{ij}\xi^1_i\xi^1_j(x_j-x_i) \Big]-2\varepsilon x_i\\
		&=\frac{1}{N}\Big[ (1+\xi^1_i\xi^2_i)\sum_{j\in I_1}(x_j-x_i)+(1-\xi^1_i\xi^2_i)\sum_{j\in I_2}(x_j-x_i)
\Big]-2\varepsilon x_i\\
		&=\frac{1}{N}\left[ -Nx_i+\xi^1_i\xi^2_i(|I_2|-|I_1|)x_i\right]-2\varepsilon x_i
\\
		&=\begin{cases}-2\left(\frac{|I_1|}{N}+\varepsilon\right)x_i , &i\in I_1,\\ 0, &i\in I_2. \end{cases}
		\end{align*}
	\end{linenomath*}
		The desired result in \eqref{equz} is proved.

		$(3)$   This  can be derived from that in $(2)$, by replacing  the set of memories  $\{\xi^1,\xi^2\}$ with $\{\xi^1,-\xi^2\}$, which produces the same Hebbian network $\{C_{ij}\}$.

		$(4)$ We will  show
		$ \mathcal J_{\xi^1}x^{\star} =-2\varepsilon x^{\star} $  
where $x^{\star}=x^{\star}(I_1;I_2)$. By \eqref{a} and \eqref{cxx} we derive 
	 \begin{linenomath*}
	 	\begin{align*}
	& \quad	(\mathcal J_{\xi^1}x^{\star})_i=\frac{1}{N}\Big[ (1+\xi^1_i\xi^2_i)\sum_{j\in I_1}(x_j^{\star}-x_i^{\star})+(1-\xi^1_i\xi^2_i)\sum_{j\in I_2}(x_j^{\star}-x_i^{\star})
\Big]-2\varepsilon x_i^{\star}\\
	 	&=\frac{1}{N}\Big[ (1+\xi^1_i\xi^2_i)|I_1|(|I_2|-x^{\star}_i)
+(1-\xi^1_i\xi^2_i)|I_2|(-|I_1|-x^{\star}_i)
\Big]
-2\varepsilon x^{\star}_i\\
	 	&=\frac{1}{N}\Big[ 2\xi^1_i\xi^2_i|I_1||I_2|-(|I_1|+|I_2|)x^{\star}_i-\xi^1_i\xi^2_i (|I_1|-|I_2|)x^{\star}_i \Big]-2\varepsilon x^{\star}_i.
	 	\end{align*}
	 \end{linenomath*}
		\noindent  $\bullet$ If $i\in I_1$, then $\xi^1_i\xi^2_i=1$ and $x^{\star}_i=|I_2|$. So we have
		\begin{linenomath*}	\[	(\mathcal J_{\xi^1}x^{\star})_i=\frac{1}{N}\Big[2 |I_1||I_2|-(|I_1|+|I_2|)|I_2|- (|I_1|-|I_2|)|I_2|\Big]-2\varepsilon x^{\star}_i=-2\varepsilon x^{\star}_i.\]\end{linenomath*}
		\noindent  $\bullet$  If $i\in I_2$, then $\xi^1_i\xi^2_i=-1$ and $x^{\star}_i=-|I_1|$. Hence we obtain
		\begin{linenomath*}	\[	(\mathcal J_{\xi^1}x^{\star})_i=\frac{1}{N}\Big[-2|I_1||I_2|+(|I_1|+|I_2|)|I_1|- (|I_1|-|I_2|)|I_1|\Big]-2\varepsilon x^{\star}_i=-2\varepsilon x^{\star}_i.\]\end{linenomath*}
		 The proof is completed.
	\end{proof}
	
	Next we consider the spectrum  of the Jacobian $\mathcal J_\eta$ at $\varphi^*(\eta)$ for $\eta\in\{1,-1\}^N\setminus \{\pm\xi^1,\pm\xi^2\}$, to explore the linear stability of   patterns other than the memorized ones.  
	We denote
\begin{align*}
	&I_{11}^{+}=\left\{j\in I_1\mid \xi^1_j=\xi^2_j=\eta_j=1\right\},  \qquad \,\,\,\,\, I_{11}^{-}=\left\{j\in I_1\mid \xi^1_j=\xi^2_j=\eta_j=-1\right\}, \\
	&I_{12}^{+}=\left\{j\in I_1\mid \xi^1_j=\xi^2_j=-\eta_j=1\right\},\,\qquad    I_{12}^{-}=\left\{j\in I_1\mid \xi^1_j=\xi^2_j=-\eta_j=-1\right\}, \\
&I_{21}^{+}=\left\{j\in I_2\mid \xi^1_j=-\xi^2_j=\eta_j=1\right\},\,\qquad
I_{21}^{-}=\left\{j\in I_2\mid \xi^1_j=-\xi^2_j=\eta_j=-1\right\}, \\
		&I_{22}^{+}=\left\{j\in I_2\mid \xi^1_j=-\xi^2_j=-\eta_j=1\right\},\,\,\quad    I_{22}^{-}=\left\{j\in I_2\mid \xi^1_j=-\xi^2_j=-\eta_j=-1\right\}.
	\end{align*}
and set
\begin{equation}\label{notation}I_{11}=I_{11}^{+}\cup I_{11}^{-},\quad I_{12}=I_{12}^{+}\cup I_{12}^{-},\quad I_{21}=I_{21}^{+}\cup I_{21}^{-},\quad I_{22}=I_{22}^{+}\cup I_{22}^{-}. \end{equation}
As $\xi^1\neq \pm \xi^2$,	it can be seen that $ I_{11}\cup I_{12}=I_1\neq \emptyset$ and  $I_{21}\cup I_{22}=I_2\neq \emptyset.$

	\begin{lemma}\label{eig_eta}
		Let $\{\xi^1,\xi^2\}$ be the set of memories in system \eqref{mod} and let $\eta\in\{1,-1\}^N\setminus \{\pm\xi^1,\pm\xi^2\}$. For $\mathcal{J}_{\eta}$, we have:\\ 
$(1)$  $0$ is a single eigenvalue with eigenspace $V[\mathbb{1}]$;\\
		$(2)$ $-2\varepsilon$ is a single eigenvalue  with eigenspace  $V(I_1;I_2)$;\\
		$(3)$ if   $I_{11}\neq \emptyset$, then $2\big(\frac{-|I_{11}|+|I_{12}|}{N}-\varepsilon\big)$ is an  eigenvalue with multiplicity $|I_{11}|-1$, corresponding to the eigenspace $V(I^{+}_{11};I^{-}_{11})+V(I^{+}_{11})+V(I^{-}_{11})$;\\
	$(4)$ if   $I_{12}\neq \emptyset$, then $2\big(\frac{|I_{11}|-|I_{12}|}{N}-\varepsilon\big)$ is an  eigenvalue with multiplicity $|I_{12}|-1$, corresponding to the  eigenspace $V(I^{+}_{12};I^{-}_{12})+V(I^{+}_{12})+V(I^{-}_{12})$;\\
	$(5)$ if   $I_{11}\neq \emptyset$ and $I_{12} \neq \emptyset$, then $2\big( \frac{|I_1|}{N}-\varepsilon \big)$ is a single eigenvalue with eigenspace $V(I_{11};I_{12})$;
		$(6)$ if $I_{21}\neq \emptyset$, then $2\big(\frac{|I_{22}|-|I_{21}|}{N}-\varepsilon\big)$ is an eigenvalue with multiplicity $|I_{21}|-1$, corresponding to the eigenspace $V(I^{+}_{21};I^{-}_{21})+V(I^{+}_{21})+V(I^{-}_{21})$;\\
		$(7)$ if $I_{22}\neq \emptyset$, then  $2\big(\frac{-|I_{22}|+|I_{21}|}{N}-\varepsilon\big)$ is an eigenvalue with multiplicity $|I_{22}|-1$,
		corresponding to the eigenspace $V(I^{+}_{22};I^{-}_{22})+V(I^{+}_{22})+V(I^{-}_{22})$;\\
		$(8)$ if $I_{21}\neq \emptyset$ and $I_{22} \neq \emptyset$, then $2\big( \frac{|I_2|}{N}-\varepsilon \big)$  is a single eigenvalue with eigenspace $V(I_{21};I_{22})$.
 \end{lemma}
	\begin{proof}
By \eqref{notation} and noting that $I_1\cup I_2=[N]$ and $I_1\cap I_2= \emptyset$, we find that the   dimensions of  those subspaces  stated in (1)-(8) sum to $N$. In fact, the sum of dimensions of  spaces in (3)-(5) is $|I_{1}|-1$, and the sum of dimensions of  spaces in (6)-(8) is $|I_{2}|-1$. Therefore, to verify the results in this lemma, it suffices to show that  any nonzero vector in each space   is an eigenvector of the corresponding  eigenvalue.

	$(1)$ The relation $\mathcal J_{\eta}\mathbb{1}=0\mathbb{1}$ is obvious, due to  the global phase shift invariance.

 {In order to prove (2)-(8), we first note that for} any    $x=(x_1,x_2,\dots,x_N)^{\mathrm{T}}$ in those subspaces mentioned in (2)-(8),  we have  $\sum_{\substack{j=1}}^{N}x_j=0$ and
		\begin{align}	\begin{aligned}\label{aa}
		(\mathcal J_{\eta}x)_i&=\sum_{\substack{j=1,j\neq i}}^{N}\Big(\frac{1}{N}C_{ij}\eta_i\eta_j+\frac{2\varepsilon}{N}\Big)x_j+\Big[-\frac{1}{N}\sum_{\substack{j=1,j\neq i}}^{N} C_{ij}\eta_i\eta_j-\frac{2\varepsilon(N-1)}{N}\Big]x_i\\
		&=\frac{1}{N}\sum_{j=1}^{N}C_{ij}\eta_i\eta_j (x_j-x_i)-2\varepsilon x_i.
		\end{aligned}\end{align}
We also note that
	\begin{linenomath*}
		\begin{align}\label{equ5} C_{ij}\eta_i\eta_j=(\xi^1_i\xi^1_j+\xi^2_i\xi^2_j)\eta_i\eta_j=\begin{cases}(\xi^1_i+\xi^2_i)\eta_i, &j\in I_{11},\\ -(\xi^1_i+\xi^2_i)\eta_i, &j\in I_{12},\\ \xi^1_j(\xi^1_i-\xi^2_i)\eta_i\eta_j, &j\in I_{2}, \end{cases}
		\end{align}
	\end{linenomath*}
Next we prove the statements (2)-(8).

		$(2)$
		It suffices to show $ \mathcal J_{\eta}x^{\star} =-2\varepsilon x^{\star}$  
where $x^\star=x^\star(I_1;I_2)$. 
	By \eqref{aa},	the desired result is equivalent to
		\begin{linenomath*}
			\begin{align}\label{equ1}
			\sum_{j=1}^{N}C_{ij}\eta_i\eta_j(x^{\star}_j- x^{\star}_i) \equiv 0, \;\;\text{for} \;\;i\in [N].
			\end{align}
		\end{linenomath*}
		In fact,
		\begin{linenomath*}
			\begin{align*}
			 \sum_{j=1}^{N}C_{ij}\eta_i\eta_j(x^{\star}_j- x^{\star}_i)
			&=\sum_{j=1}^{N}(\xi^1_i\xi^1_j+\xi^2_i\xi^2_j)\eta_i\eta_j (x^{\star}_j-x^{\star}_i) \\
			&=\sum_{j\in I_1}\xi^1_j(\xi^1_i+\xi^2_i)\eta_i\eta_j (x^{\star}_j-x^{\star}_i) +\sum_{j\in I_2}\xi^1_j(\xi^1_i-\xi^2_i)\eta_i\eta_j (x^{\star}_j-x^{\star}_i).
			\end{align*}
		\end{linenomath*}
		If $i\in I_1$, then $\xi^1_i-\xi^2_i=0$ and $x^{\star}_j-x^{\star}_i=0$ for $j\in I_1$.  If $i\in I_2$, then $\xi^1_i+\xi^2_i=0$ and $x^{\star}_j-x^{\star}_i=0$ for $j\in I_2$. So the equality \eqref{equ1} holds.

	$(3)$
	  It suffices to prove that the nonzero vectors in  $V(I^{+}_{11};I^{-}_{11}),V(I^{+}_{11})$ and $V(I^{-}_{11})$ are eigenvectors corresponding to the eigenvalue $2\big(\frac{-|I_{11}|+|I_{12}|}{N}-\varepsilon\big)$.
	
	For $V(I^{+}_{11};I^{-}_{11})$,
	we need to prove $(\mathcal J_{\eta}x^{\star})_i=2\big(\frac{-|I_{11}|+|I_{12}|}{N}-\varepsilon\big) x^{\star}_i $ for $i\in [N]$ where $x^\star=x^\star(I_{11}^{+};I_{11}^{-})$. 
	This is equivalent to show
	\begin{linenomath*}
		\begin{align}
		&(\mathcal J_{\eta}x^{\star})_i=2\big(\frac{-|I_{11}|+|I_{12}|}{N}-\varepsilon\big) |I_{11}^{-}|,\;\;\text{for} \;\;i\in I_{11}^{+},\label{equ6}\\
		&(\mathcal J_{\eta}x^{\star})_i=-2\big(\frac{-|I_{11}|+|I_{12}|}{N}-\varepsilon\big) |I_{11}^{+}|,\;\;\text{for} \;\;i\in I_{11}^{-},\label{equ7}\\
		&(\mathcal J_{\eta}x^{\star})_i=0,\;\;\text{for} \;\;i\notin I_{11}.\label{equ8}
		\end{align}
	\end{linenomath*}
 By \eqref{aa} and \eqref{equ5},  we obtain
	\begin{linenomath*}
		\begin{align*}
			(\mathcal J_{\eta}x^{\star})_i&=\frac{1}{N}\Big(\sum_{j\in I_{11}^{+}}+\sum_{j\in I_{11}^{-}} +\sum_{j\in I_{12}} +\sum_{j\in I_{2}} \Big)C_{ij}\eta_i\eta_j (x^{\star}_j-x^{\star}_i)-2\varepsilon x^{\star}_i\\
		&=\frac{1}{N}(\xi^1_i+\xi^2_i)\eta_i\sum_{j\in I_{11}^+}(x^{\star}_j-x^{\star}_i)+\frac{1}{N}(\xi^1_i+\xi^2_i)\eta_i\sum_{j\in I_{11}^{-}}(x^{\star}_j-x^{\star}_i)\\
		&\quad -\frac{1}{N}(\xi^1_i+\xi^2_i)\eta_i\sum_{j\in I_{12}}(x^{\star}_j-x^{\star}_i) +\frac{1}{N}(\xi^1_i-\xi^2_i)\eta_i\sum_{j\in I_{2}}\xi^1_j\eta_j(x^{\star}_j-x^{\star}_i)-2\varepsilon x^{\star}_i\\
		& {=\frac{1}{N}(\xi^1_i+\xi^2_i)\eta_i  |I_{11}^+|(|I_{11}^-|-x^{\star}_i)+\frac{1}{N}(\xi^1_i+\xi^2_i)\eta_i | I_{11}^{-}|(-|I_{11}^+|-x^{\star}_i)}\\
		&\quad  {+\frac{1}{N}(\xi^1_i+\xi^2_i)\eta_i|I_{12}| x^{\star}_i  -\frac{1}{N}(\xi^1_i-\xi^2_i)\eta_i\sum_{j\in I_{2}}\xi^1_j\eta_j  x^{\star}_i -2\varepsilon x^{\star}_i}
		\end{align*}
	\end{linenomath*}
	 {If $i\in I_{11}^+$, then $(\xi^1_i+\xi^2_i)\eta_i=2$, $\xi^1_i-\xi^2_i=0$ and $x^{\star}_i=|I_{11}^-|$; and combining $|I_{11}^+|+|I_{11}^{-}|=|I_{11}|$, we have \eqref{equ6}. If $i\in I_{11}^{-}$, then $(\xi^1_i+\xi^2_i)\eta_i=2$, $\xi^1_i-\xi^2_i=0$ and $x^{\star}_i=-|I_{11}^{+}|$; so the equality \eqref{equ7} holds. If $i\notin I_{11}$, then 
$x^{\star}_i=0$, 
which implies  \eqref{equ8}.}
	
	Secondly, we consider the vectors in $V(I^{+}_{11})$. For any $x=(x_1,x_2,\dots,x_N)^{\mathrm{T}}\in V(I^{+}_{11})$,
	we need to prove
	\begin{linenomath*}
		\begin{align}
		&(\mathcal J_{\eta}x)_i=0,\;\;\text{for} \;\;i\notin I_{11}^{+},\label{equ9}\\
		&(\mathcal J_{\eta}x)_i=2\big(\frac{-|I_{11}|+|I_{12}|}{N}-\varepsilon\big) x_i,\;\;\text{for} \;\;i\in I_{11}^{+}.\label{equ10}
		\end{align}
	\end{linenomath*}
	By relations \eqref{aa} and \eqref{equ5}, we obtain
	\begin{linenomath*}
		\begin{align*}
		&	(\mathcal J_{\eta}x)_i=\frac{1}{N}\Big(\sum_{j\in I_{11}}+\sum_{j\in I_{12}}+\sum_{j\in I_{2}}\Big)C_{ij}\eta_i\eta_j (x_j-x_i) -2\varepsilon x_i\\
		&=\frac{1}{N}(\xi^1_i+\xi^2_i)\eta_i\sum_{j\in I_{11}}(x_j-x_i)-\frac{1}{N}(\xi^1_i+\xi^2_i)\eta_i\sum_{j\in I_{12}}(x_j-x_i)+\frac{1}{N}(\xi^1_i-\xi^2_i)\eta_i\sum_{j\in I_{2}}\xi^1_j\eta_j(x_j-x_i)\\
		&\quad
		 -2\varepsilon x_i.
		\end{align*}
	\end{linenomath*}
   Using the definition of $V(I^{+}_{11})$, we obtain $\sum_{j\in I_{11}^{+}}x_j=0$ and $x_j=0$ for $j\notin I_{11}^{+}$. Then
   \[	(\mathcal J_{\eta}x)_i=\frac{1}{N}(\xi^1_i+\xi^2_i)\eta_i(-|I_{11}|+|I_{12}|)x_i-\frac{1}{N}(\xi^1_i-\xi^2_i)\eta_i\sum_{j\in I_{2}}\xi^1_j\eta_jx_i
   -2\varepsilon x_i.\]
	If $i\notin I_{11}^+$, then $x_i=0$ and the equality \eqref{equ9} holds. If $i\in I_{11}^{+}$, then $(\xi^1_i+\xi^2_i)\eta_i=2$ and $\xi^1_i-\xi^2_i=0$, which imply   \eqref{equ10}.
	
	The proof for $V(I^{-}_{11})$ is similar to that for $V(I_{11}^{+})$, so it is omitted here.
	
    $(4)$  The desired result can be derived from that in $(3)$, by replacing  the   patterns $\eta$ by $-\eta$, which has the same Jacobian, i.e., $\mathcal J_{-\eta}=\mathcal J_{\eta}$.

     $(5)$ It suffices to show $(\mathcal J_{\eta}x^{\star})_i=2\big(\frac{|I_1|}{N}-\varepsilon\big) x^{\star}_i$ for $i\in [N]$ where $x^\star=x^\star(I_{11}; I_{12})$. 
     This is equivalent to
     \begin{linenomath*}
     	\begin{align}
     	&(\mathcal J_{\eta}x^{\star})_i=2\big(\frac{|I_1|}{N}-\varepsilon\big) |I_{12}|,\;\;\text{for} \;\;i\in I_{11},\label{equ2}\\
     	&(\mathcal J_{\eta}x^{\star})_i=-2\big(\frac{|I_1|}{N}-\varepsilon\big) |I_{11}|,\;\;\text{for} \;\;i\in I_{12},\label{equ3}\\
     	&(\mathcal J_{\eta}x^{\star})_i=0,\;\;\text{for} \;\;i\in I_{2}.\label{equ4}
     	\end{align}
     \end{linenomath*}
     By the relations \eqref{aa} and \eqref{equ5},
     we obtain
     \begin{linenomath*}
     	\begin{align*}
     	&	(\mathcal J_{\eta}x^{\star})_i=\frac{1}{N}\Big(\sum_{j\in I_{11}}+\sum_{j\in I_{12}}+\sum_{j\in I_{2}}\Big)C_{ij}\eta_i\eta_j (x^{\star}_j-x^{\star}_i)-2\varepsilon x^{\star}_i\\
     	&=\frac{1}{N}(\xi^1_i+\xi^2_i)\eta_i\sum_{j\in I_{11}}(x^{\star}_j-x^{\star}_i)-\frac{1}{N}(\xi^1_i+\xi^2_i)\eta_i\sum_{j\in I_{12}}(x^{\star}_j-x^{\star}_i)+\frac{1}{N}(\xi^1_i-\xi^2_i)\eta_i\sum_{j\in I_{2}}\xi^1_j\eta_j(x^{\star}_j-x^{\star}_i)\\
     	&\quad
     	-2\varepsilon x^{\star}_i\\
     & {=\frac{1}{N}(\xi^1_i+\xi^2_i)\eta_i|I_{11}|(|I_{12}|-x^{\star}_i)
     -\frac{1}{N}(\xi^1_i+\xi^2_i)\eta_i|I_{12}|(-|I_{11}|-x^{\star}_i)
     -\frac{1}{N}(\xi^1_i-\xi^2_i)\eta_i\sum_{j\in I_{2}}\xi^1_j\eta_jx^{\star}_i}\\
     	&\quad
     	-2\varepsilon x^{\star}_i.
     	\end{align*}
     \end{linenomath*}
    {If $i\in I_{11}$, then $(\xi^1_i+\xi^2_i)\eta_i=2$, $\xi^1_i-\xi^2_i=0$ and $x^{\star}_i=|I_{12}|$; and combing $|I_{11}|+|I_{12}|=|I_1|$, we derive \eqref{equ2}. If $i\in I_{12}$, then $(\xi^1_i+\xi^2_i)\eta_i=-2$, $\xi^1_i-\xi^2_i=0$ and $x^{\star}_i=-|I_{11}|$; so the equality \eqref{equ3} holds. If $i\in I_{2}$, then $\xi^1_i+\xi^2_i=0$ and $x^{\star}_i=0$, 
     which imply \eqref{equ4}.}

     (6)-(8)  Notice that for binary patterns,  $\xi^2$ and $-\xi^2$ can be regarded as   the same pattern.  In fact, if the memorized  pattern  $\xi^2$ is replaced by $-\xi^2$,  the coupling term $C_{ij}$ does not change.
		  Replacing memories $\{\xi^1,\xi^2\}$ with $\{\xi^1,-\xi^2\}$ and applying  (3)-(5)  immediately yield  the desired results in (6)-(8).
	\end{proof}

\begin{remark}	 {
{\em If   $I_{11}= \emptyset$ or $I_{12} =\emptyset$, then $V(I_{11};I_{12})=\{\mathbb 0\}$  and in fact, $2\big( \frac{|I_1|}{N}-\varepsilon \big)$ does not show up in the eigenspectrum of $\mathcal J_\eta$.  For example, let's say $I_{11}= \emptyset$ and $I_1=I_{12} \neq\emptyset$, then Lemma \ref{eig_eta} (4) tells  that $2\big(\frac{|I_{11}|-|I_{12}|}{N}-\varepsilon\big)$ is an  eigenvalue with multiplicity $|I_{1}|-1$.
	Similarly,  if   $I_{21}=\emptyset$ or $I_{22}= \emptyset$, then $2\big( \frac{|I_2|}{N}-\varepsilon \big)$  does not show up in the eigenspectrum.}}

 {{\em Therefore,   the Jacobian   $\mathcal J_{\eta}$   of system \eqref{mod} at $\varphi^*(\eta)$   has at most $8$ distinct  eigenvalues, of which the maximum one is at most
		\begin{linenomath*}\[ 
		\max\Big\{2\Big(\frac{|I_1|}{N}-\varepsilon\Big), 2\Big(\frac{|I_2|}{N}-\varepsilon\Big), 0 \Big\}. \]\end{linenomath*}
	In fact, $-2\varepsilon<0$ and
	$\max\{|I_{11}|-|I_{12}|,|I_{12}|-|I_{11}|\}<|I_1|$ and $ \max\{|I_{22}|-|I_{21}|,|I_{21}|-|I_{22}|\}<|I_2|,$
	where we used    $I_1=I_{11}\cup I_{12}$ and $I_2=I_{21}\cup I_{22}.$ }}
	\end{remark}

	The following theorem is the main result of this paper.
	\begin{theorem}\label{onlytwo}
		Let $\{\xi^1,\xi^2\}$ be the set of memories in system \eqref{mod} {satisfying $1<|I_1|<N-1$} and let $\varepsilon^*=\min\Big\{\frac{|I_1|}{N},\frac{|I_2|}{N}\Big\}$, then we have:\\
		$(1)$ the patterns $\pm\xi^1$ and $\pm\xi^2$ are $\varepsilon$-independently asymptotically stable;\\
		$(2)$ if $\varepsilon\in(0,\varepsilon^*)$,  all   binary patterns
are  unstable except for  $ \pm\xi^1$ and $\pm\xi^2$; \\
			$(3)$ if $\varepsilon\in  {[\varepsilon^*,\infty)}$,  {then there exists at least a stable binary pattern    $ \eta\in \{1,-1\}^N\setminus\{\pm\xi^1,\pm\xi^2\}$.}
	\end{theorem}
	\begin{proof}
		$(1)$  This follows directly from Proposition \ref{xi12}, since  the eigenspectrum of the Jacobian at $\xi^1$ or $\xi^2$ consists of  $N-1$ negative reals  and one zero (due to global phase shift invariance).
		
		$(2)$ For any $\eta\in \{1,-1\}^N\setminus\{\pm\xi^1, \pm\xi^2\}$, we verify its  instability by proving that at least one of $2\big(\frac{|I_1|}{N}-\varepsilon\big)$ and $2\big(\frac{|I_2|}{N}-\varepsilon\big)$, shows up in the eigenspectrum of $\mathcal J_\eta$,  in Lemma \ref{eig_eta}. If   $2\big(\frac{|I_1|}{N}-\varepsilon\big)$ does not show up, then at least one of $I_{11}$ and $I_{12}$ is empty; without loss of generality, let's say $I_{11}=\emptyset$.  {This will lead to $I_{21}\neq\emptyset$ and $I_{22}\neq\emptyset$. Otherwise,    $I_{11}=I_{21}=\emptyset$ implies  $I_{12}\cup I_{22}=[N]$ and $\eta=-\xi^1$; $I_{11}=I_{22}=\emptyset$ implies  $I_{12}\cup I_{21}=[N]$ and $\eta=-\xi^2$; this contradicts to the setting $\eta\notin \{\pm\xi^1,\pm\xi^2\}$. Therefore, we have $I_{21}\neq\emptyset$ and $I_{22}\neq\emptyset$, which imply that the eigenvalue $2\big(\frac{|I_2|}{N}-\varepsilon\big)$  shows up. Now  for $\varepsilon\in(0,\varepsilon^*)$, we have  $2\big(\frac{|I_2|}{N}-\varepsilon\big)>0$ and $   2\big(\frac{|I_1|}{N}-\varepsilon\big)>0.$  This implies the instability of $\eta$.}
		
		$(3)$
		   Without loss of generality,   we say $\frac{|I_2|}{N}\ge\frac{|I_1|}{N}$ which implies $\varepsilon^*=\frac{|I_1|}{N}$.
		 If $\varepsilon>\varepsilon^*$, then we have   $\frac{|I_1|}{N}-\varepsilon<0$.
		   Let $\eta$ be a binary pattern satisfying $I_{21}= \{j  \,\big|\,\xi^1_j=-\xi^2_j=\eta_j  \}=\emptyset$,  {$I_{11}= \{j  \,\big|\,\xi^1_j=\xi^2_j=\eta_j  \}\neq\emptyset$ and $I_{12}= \{j  \,\big|\,\xi^1_j=\xi^2_j=-\eta_j  \}\neq\emptyset$
	   (such $\eta$ must exist since $1<|I_1|<N-1$).}
		Then $2\big(\frac{|I_2|}{N}-\varepsilon\big)$ does not show up in the eigenspectrum of $\mathcal J_{\eta}$, while $2\big(\frac{|I_1|}{N}-\varepsilon\big)$ is an eigenvalue and 
		$\lambda_{\max}(\mathcal J_\eta)=\max\big\{2\big(\frac{|I_1|}{N}-\varepsilon\big), 0 \big\}.$  Here,  $0$ is a single eigenvalue   due to the global phase shift invariance.
		Therefore, the pattern $\eta$ is  asymptotically stable. This also implies that  $\varphi^*(\eta)$ is a strict local minimum of the potential $f_{\varepsilon}$ in \eqref{potential}.  If $\varepsilon=\varepsilon^*$,   by the continuity of $f_{\varepsilon}(\varphi)-f_{\varepsilon}(\varphi^*)$ with respect to $\varepsilon$,   we see  that  $\varphi^*(\eta)$ is a local minimum of $f_{\varepsilon^*}$, so  $\eta$ is stable.
	\end{proof}
	\begin{remark}
	 {In the proof of  Theorem \ref{onlytwo} (3) with the condition $\frac{|I_2|}{N}\ge\frac{|I_1|}{N}$, we find that when $\varepsilon=\varepsilon^*$,
		the number of stable binary pattern $\eta\in \{1,-1\}^N\setminus\{\pm\xi^1,\pm\xi^2\}$ is at least $2^{|I_1|}-1$. In fact, by the proof of (3), any $\eta$ satisfying $I_{21}= \emptyset$,  $I_{11} \neq\emptyset$ and $I_{12} \neq\emptyset$ is stable.    Obviously, the number of  such  $\eta$ is at least  $ \tbinom{|I_1|}{1}  +\tbinom{|I_1|}{2}+\dots+\tbinom{|I_1|}{|I_1|-1} =2^{|I_1|}-1 $. } 
	\end{remark}
	
 {If $|I_1|=1$ or $|I_1|=N-1$, the critical strength is different with the case in Theorem \ref{onlytwo}. The results   are summarized in the following proposition.}
	\begin{proposition}\label{I1}
	 {\em	Let $\{\xi^1,\xi^2\}$ be a set of memories in system \eqref{mod} satisfying $|I_1|=1$, and let $\varepsilon^{**}:=\max\big\{\frac{|I_1|}{N},\frac{|I_2|}{N}\big\}=\frac{N-1}{N}.$ We have:\\
		(1) the patterns $\pm\xi^1$ and $\pm\xi^2$ are $\varepsilon$-independently asymptotically stable;\\
		(2) for any binary pattern  $ \eta\in \{1,-1\}^N\setminus\{\pm\xi^1,\pm\xi^2\}$, it is unstable if $\varepsilon\in(0,\varepsilon^{**})$ and stable if {$\varepsilon\in [\varepsilon^{**},\infty)$}.}\end{proposition}
		\begin{proof}
The first assertion  follows directly from Proposition \ref{xi12} again. For the second assertion, by $|I_1|=1$ and the definitions of  $I_{11}$ and $I_{12}$, we see that $I_{11}$ or $I_{12}$ is   empty. By Lemma \ref{eig_eta} (5), $2(\frac{|I_1|}{N}-\varepsilon)$ does not show up in the eigenspectrum of the Jacobian $\mathcal{J}_{\eta}$. By $ \eta\notin \{\pm\xi^1,\pm\xi^2\}$ and Lemma \ref{eig_eta} (8), then $I_{21}\neq\emptyset,I_{22}\neq\emptyset$ and the eigenvalue $2(\frac{|I_2|}{N}-\varepsilon)$ of $\mathcal{J}_{\eta}$ must   appear. So  $\lambda_{\max}(\mathcal{J}_{\eta})=\max\big\{2(\frac{|I_2|}{N}-\varepsilon),0\big\}$, where $0$ is a single eigenvalue   due to the global phase shift invariance. If $\varepsilon\in(0,\varepsilon^{**})$, then $\lambda_{\max}(\mathcal{J}_{\eta})>0$ and $\eta$ is unstable.  {If $\varepsilon\in (\varepsilon^{**},\infty)$, then $2(\frac{|I_2|}{N}-\varepsilon)<0$ and $\eta$ is asymptotically stable. By the same argument as in the proof of Theorem \ref{onlytwo} (3),  $\eta$ is stable if  $\varepsilon=\varepsilon^{**}$. }
	\end{proof}
	
	\begin{remark}\label{remark}
		{\em In \cite{L-Z-X,Z-L-X}, the system \eqref{mod}  with a set of mutually orthogonal memorized binary patterns $\{\xi^1,\xi^2,\dots,\xi^M \}$ was considered. It was proved that the memories are $\varepsilon$-independently stable, and there can be some additional $\varepsilon$-independently   stable patterns in addition to the memorized ones if  $M\geq 4$. 
		For those patterns $\eta$   lacking $\varepsilon$-independent  stability, there exists a critical $\varepsilon_\eta^*>0$ such that $\eta$ is stable for $\varepsilon\in (\varepsilon_\eta^*,+\infty)$ and unstable for $\varepsilon\in (0, \varepsilon_\eta^*)$.  The critical strength $\varepsilon_\eta^*$ depends on the set $\{\xi^1,\xi^2,\dots,\xi^M,\eta \}$ 
but there is no formula for its exact value. Instead, a lower bound  is given by 
		\begin{linenomath*}
		\begin{equation*}
	\varepsilon_\eta^*\geq\max_{l\in\{1,2,\dots,M\}}\frac{N^{2}-\sum_{k=1}^{M}(\xi^{k}\cdot \eta)^2}{2(N^{2}-(\xi^{l}\cdot \eta)^2)}=:\varepsilon_{\eta}.
		\end{equation*}
		\end{linenomath*}
		Therefore, in practice  we are able to  distinguish  the memories  from those patterns $\eta$'s which lack $\varepsilon$-independent  stability,  by using  a choice of  $\varepsilon\in (0, \varepsilon_\eta)$.   However, there are some drawbacks:  (1) the theory is valid only for memories with mutual orthogonality;  (2)    the  bound $\varepsilon_\eta$ needs computation and is  conservative, which  can be rather small; (3) the existence of additional $\varepsilon$-independently stable patterns may break the error-free retrieval since other patterns can emerge in the evolution.
		These reflect  the advantages of Theorem \ref{onlytwo} which gives the sharp critical
		strength to distinguish the memories from all other patterns.}  
	\end{remark}
	
	
\subsection{Basin of stable patterns}\label{sec_domain}
	Theorem \ref{onlytwo} tells  that the  equilibria  $\varphi^*(\xi^1)$ and $\varphi^*(\xi^2)$ are  $\varepsilon$-independently stable.  If $\varepsilon\in(0,\varepsilon^*)$,  $\varphi^*(\eta)$ corresponding to the binary pattern $\eta\notin\{\pm \xi^1,\pm\xi^2\}$ are unstable. In this subsection we  consider  the basin of $\varphi^*(\xi^k),k=1,2$.
	
	For $k=1,2$, let $\tilde{\varphi}_{i}(t)=\varphi_{i}(t)-\varphi_{i}^{*}(\xi^k)$, then we see by \eqref{mod}
	\begin{align}\label{tilde_model}
	\dot{\tilde\varphi}_i  = \frac{1}{N}\sum_{j=1}^{N}\left(C_{ij}\xi_{i}^{k}\xi_{j}^{k}+2\varepsilon  \cos(\tilde\varphi_{j}  -\tilde\varphi_{i}  )\right)
	\sin(\tilde\varphi_{j}   - \tilde\varphi_i  ) .
	\end{align}
	Then the equilibrium   $\tilde\varphi^{*}=(0,0,\dots,0)$ is $\varepsilon$-independently stable for \eqref{tilde_model}.  
We will consider   the basin   of $\tilde\varphi^{*}$ in  system \eqref{tilde_model} and investigate the asymptotic collective behavior, which reflect  the basin of $\varphi^{*}(\xi^k)$ for system \eqref{mod}. 
	
	It is known that for $k=1,2$,
	\begin{align*}
	C_{ij}\xi^k_i\xi^k_j &=\left( \xi^1_i\xi^1_j +\xi^2_i\xi^2_j  \right)\xi^k_i\xi^k_j
	=1+\xi^1_i\xi^1_j \xi^2_i\xi^2_j
	=\begin{cases}2, &i\in I_1, j\in I_1\;\; \text{or} \;\; i\in I_2, j\in I_2 ;\\0,&i\in I_1, j\in I_2\;\; \text{or}\;\;  i\in I_2, j\in I_1. \end{cases}
	\end{align*}
	For any $i,j\in [N]$, define
	\begin{align*}
 {	K_{ij}(t)=C_{ij}\xi_{i}^{k}\xi_{j}^{k}+2\varepsilon\cos(\tilde\varphi_{j}(t)-\tilde\varphi_{i}(t))\quad \text{and}\quad \mathcal{D}(\tilde{\varphi}(t))=\max_{i,j\in [N]}(\tilde{\varphi}_i(t)-\tilde{\varphi}_j(t)),}
	\end{align*}
	then $K_{ij}(t)=K_{ji}(t)$ and $K_{ij}(t)\ge 2\varepsilon\cos(\tilde\varphi_{j}(t)-\tilde\varphi_{i}(t)).$
	
	\begin{lemma}\label{lem1}
		Let $\tilde\varphi(t)=(\tilde\varphi_1(t),\tilde\varphi_2(t),\dots,\tilde\varphi_N(t))$ be a solution to system \eqref{tilde_model} and $T>0$. If
		 $ \sup_{t\in[0,T)} \mathcal{D}(\tilde{\varphi}(t))<H$
	for some $H\in \left(0,\frac{\pi}{2}\right),$	then we have
		\begin{align*}
		\mathcal{D}(\tilde{\varphi}(t))\le \mathcal{D}(\tilde{\varphi}(0)) e^{-\lambda _1 t},\;\quad  t\in[0,T),
		\end{align*}
		where $\lambda_1$ is given by  $\lambda_1=\frac{4\varepsilon \cos H \cos \frac{H}{2}}{\pi}>0$.
	\end{lemma}
\begin{proof}
 {
	Let
 {\[I_{\max}(t):=\argmax_{j}\{\tilde{\varphi}_j(t)\}\quad \text{and}\quad I_{\min}(t):=\argmin_{j}\{\tilde{\varphi}_j(t)\}.\]}
Following  [\cite{L-F-M}, Lemma 2.2],  the upper Dini derivative of $\mathcal{D}(\tilde{\varphi}(t))$ along the system \eqref{tilde_model} is given by
	\[D^{+}\mathcal{D}(\tilde{\varphi}(t))=\limsup\limits_{h\downarrow 0}\frac{\mathcal{D}(\tilde{\varphi}(t+h))-\mathcal{D}(\tilde{\varphi}(t))}{h}=\dot{\tilde{\varphi}}_{\bar{M}}(t)-\dot{\tilde{\varphi}}_{\bar{m}}(t),\]
	where $\bar{M}$ and $\bar{m}$ are indices which have the properties that
	\[\dot{\tilde{\varphi}}_{\bar{M}}(t)=\max\{\dot{\tilde{\varphi}}_{M}(t)\mid M\in I_{\max}(t)\}\quad \text{and}\quad \dot{\tilde{\varphi}}_{\bar{m}}(t)=\min\{\dot{\tilde{\varphi}}_{m}(t)\mid m\in I_{\min}(t)\}.\]
	For any $t\in[0,T)$,  we have $K_{ij}(t)\ge 2\varepsilon\cos  H>0$ and
	\[ -\frac{H}{2}\le-\frac{\mathcal{D}(\tilde{\varphi}(t))}{2}\le \frac{\tilde{\varphi}_j(t)-\tilde{\varphi}_{\bar{M}}(t)}{2} \le 0 \quad \text{and}\quad 0\le \frac{\tilde{\varphi}_j(t)-\tilde{\varphi}_{\bar{m}}(t)}{2}\le \frac{\mathcal{D}(\tilde{\varphi}(t))}{2}\le \frac{H}{2},\]
	and therefore,
	\begin{align*}
&\quad 	\frac{1}{2}D^{+}\mathcal{D}^2(\tilde{\varphi}(t))=\mathcal{D}(\tilde{\varphi}(t))D^{+}\mathcal{D}(\tilde{\varphi}(t))=\mathcal{D}(\tilde{\varphi}(t))(\dot{\tilde{\varphi}}_{\bar{M}}(t)-\dot{\tilde{\varphi}}_{\bar{m}}(t))\\
	&=\mathcal{D}(\tilde{\varphi}(t))\Big[ \frac{1}{N}\sum_{j=1}^{N}K_{\bar{M}j}\sin (\tilde{\varphi}_j(t)-\tilde{\varphi}_{\bar{M}}(t))-\frac{1}{N}\sum_{j=1}^{N}K_{\bar{m}j}\sin (\tilde{\varphi}_j(t)-\tilde{\varphi}_{\bar{m}}(t)) \Big]\\
	&\le \mathcal{D}(\tilde{\varphi}(t))\Big[ \frac{1}{N}\sum_{j=1}^{N}2\varepsilon\cos H\sin (\tilde{\varphi}_j(t)-\tilde{\varphi}_{\bar{M}}(t))-\frac{1}{N}\sum_{j=1}^{N}2\varepsilon\cos H\sin (\tilde{\varphi}_j(t)-\tilde{\varphi}_{\bar{m}}(t)) \Big]\\
	&=\frac{2\varepsilon\cos H}{N}\mathcal{D}(\tilde{\varphi}(t)) \sum_{j=1}^{N}\left(\sin (\tilde{\varphi}_j(t)-\tilde{\varphi}_{\bar{M}}(t))-\sin (\tilde{\varphi}_j(t)-\tilde{\varphi}_{\bar{m}}(t))\right).
	\end{align*}
	We use the inequality  $x\sin x\ge \frac{2}{\pi}x^2$ for  $|x|\le \frac{\pi}{2}$ 
to derive
		\begin{align*}
	&\quad 	\frac{1}{2}D^{+}\mathcal{D}^2(\tilde{\varphi}(t))\\
	&=\frac{2\varepsilon\cos H}{N}\mathcal{D}(\tilde{\varphi}(t)) \sum_{j=1}^{N}\Big[ 2\cos\Big(\frac{\tilde{\varphi}_j(t)-\tilde{\varphi}_{\bar{M}}(t)}{2}+
\frac{\tilde{\varphi}_j(t)-\tilde{\varphi}_{\bar{m}}(t)}{2}\Big) \sin\Big(-\frac{\tilde{\varphi}_{\bar{M}}(t)-\tilde{\varphi}_{\bar{m}}(t)}{2}\Big)\Big]\\
	&\le \frac{2\varepsilon\cos H}{N}\mathcal{D}(\tilde{\varphi}(t)) \sum_{j=1}^{N}\Big[2\cos \frac{H}{2}\sin\Big(-\frac{\tilde{\varphi}_{\bar{M}}(t)-\tilde{\varphi}_{\bar{m}}(t)}{2}\Big)\Big]\\
	&=-4\varepsilon\cos H\cos \frac{H}{2}\mathcal{D}(\tilde{\varphi}(t))\sin\frac{\mathcal{D}(\tilde{\varphi}(t))}{2}\\
	&\le -8\varepsilon\cos H\cos \frac{H}{2}\frac{2}{\pi}\frac{\mathcal{D}^2(\tilde{\varphi}(t))}{4}=-\lambda_1 \mathcal{D}^2(\tilde{\varphi}(t)), \,\quad t\in [0,T).
	\end{align*}
	Finally, using the Gronwall's inequality, we obtain
	$\mathcal{D}^{2}(\tilde{\varphi}(t))\le \mathcal{D}^{2}(\tilde{\varphi}(0)) e^{-2\lambda _1 t},\; t\in[0,T);$ that is,
	$\mathcal{D}(\tilde{\varphi}(t))\le \mathcal{D}(\tilde{\varphi}(0)) e^{-\lambda _1 t},\; t\in[0,T).$
}\end{proof}

\begin{theorem}\label{thm1}
	Let $\tilde\varphi(t)=(\tilde\varphi_1(t),\tilde\varphi_2(t),\dots,\tilde\varphi_N(t))$ is a solution of system \eqref{tilde_model} . If the initial data satisfies
	$\mathcal{D}(\tilde{\varphi}(0))<\frac{\pi}{2},$
	then we have: \\
	$(1)$ $\sup_{t\in[0,\infty)} \mathcal{D}(\tilde{\varphi}(t))<\frac{\pi}{2};$\\
	$(2)$ there exists $\lambda>0$ such that $\mathcal{D}(\tilde{\varphi}(t))\le \mathcal{D}(\tilde{\varphi}(0))e^{-\lambda t},\; t\in[0,\infty).$
\end{theorem}
\begin{proof}
 {	$(1)$ Let $H$ be a constant with $\mathcal{D}(\tilde{\varphi}(0))< H <\frac{\pi}{2}$. Define the set  $\mathcal T$ and its supremum:
	\begin{align*}
	\mathcal T=\left\{T\in[0,+\infty)\,\Big|\,\mathcal{D}(\tilde\varphi(t))<H, \,\, \text{for any} \,t\in [0,T) \right\},\quad T^*=\sup \mathcal T.
	\end{align*}
	Note that   $\mathcal{D}(\tilde\varphi(0))<H$, and $\mathcal{D}(\tilde\varphi(t))$ is a continuous function of $t$, 
the set $\mathcal T$ is nonempty, and $T^*$ is well-defined. We now claim that $T^*=+\infty.$ If not, i.e., $T^*<+\infty$, then from the continuity of $\mathcal{D}(\tilde\varphi(t))$, we have
	\begin{align*}
	\mathcal{D}(\tilde\varphi(t))<H,\;t\in[0,T^*)\quad  \text{and} \quad  \mathcal{D}(\tilde\varphi(T^*))=H.
	\end{align*}
	We use Lemma \ref{lem1} to obtain
	$ \mathcal{D}(\tilde{\varphi}(t))\le \mathcal{D}(\tilde{\varphi}(0))e^{-\lambda_1 t},\; t\in [0,T^*). $
	Let $t\to T^{*-}$, then $\mathcal{D}(\tilde\varphi(T^*))\le \mathcal{D}(\tilde{\varphi}(0))<H$, which is contradictory to $ \mathcal{D}(\tilde\varphi(T^*))=H$. Therefore, we have $T^*=\infty.$
	i.e.,  $\mathcal{D}(\tilde\varphi(t))< H,\;  t\in [0,\infty).$\\
	$(2)$ Combing $(1)$ and Lemma \ref{lem1}, we have $\mathcal{D}(\tilde{\varphi}(t))\le \mathcal{D}(\tilde{\varphi}(0))e^{-\lambda_1 t},\; t\in[0,\infty).$}
\end{proof}
	

\begin{remark}
	\em Theorem \ref{thm1} gives a  nontrivial estimate on the basin of stable memories. It tells that, if the initial data  reflecting a defective  pattern (a grayscale image), is not too far away from some memory, then the memory will be retrieved.    As a first rigorous analysis on the basin, we acknowledge that the estimate in Theorem \ref{thm1} is conservative.  For example, the defective patterns obtained by flipping the sign in certain components of   memories cannot be included in the estimated  basin. Further seeking on the basin of stable patterns will be an interesting future problem.
\end{remark}

	\section{The system with three memories}\label{sec3} \setcounter{equation}{0}
	In this section,  we are going to show  that the   theory in subsection \ref{subsec1}, for the case of two memories, cannot be extended to the case of three memories. 
We will prove that the memories are not $\varepsilon$-independently stable in a weak condition.
	
	In this section, we consider the system \eqref{mod1} with a set of  memories $\{\xi^{1},\xi^2,\xi^3\}\subset \{1,-1\}^N$, written as
\begin{linenomath*}
			\begin{equation}\label{mod3}
		{\dot \varphi}_i =\frac{1}{N}\sum_{j=1}^{N}
		S_{ij}\sin(\varphi_{j} - \varphi_i )+\frac{\varepsilon}{N}\sum_{j=1}^{N}\sin2(\varphi_{j} -\varphi_{i}),  \quad  S_{ij}:= \xi^1_i\xi^1_j+\xi^2_i\xi^2_j+\xi^3_i\xi^3_j.
		\end{equation}
	\end{linenomath*}
Again we suppose  $\xi^k\neq \pm\xi^l (k\neq l)$ and  the Jacobian of \eqref{mod3} at a binary pattern $\eta$ is
\begin{linenomath*}
	\[\mathcal I_{\eta}=\begin{pmatrix}
	-\hat  T_{1} & \frac{1}{N}S_{12}\eta_{1}\eta_{2}+\frac{2\varepsilon}{N} & \frac{1}{N}S_{13}\eta_{1}\eta_{3}+\frac{2\varepsilon}{N} & \dots &\frac{1}{N}S_{1N}\eta_{1}\eta_{N}+\frac{2\varepsilon}{N}  \\
	\frac{1}{N}S_{21}\eta_{2}\eta_{1}+\frac{2\varepsilon}{N} & -\hat T_{2} & \frac{1}{N}S_{23}\eta_{2}\eta_{3}+\frac{2\varepsilon}{N} & \dots & \frac{1}{N}S_{2N}\eta_{2}\eta_{N}+\frac{2\varepsilon}{N} \\
	\dots & \dots & \dots & \ddots & \dots \\
	\frac{1}{N}S_{N1}\eta_{N}\eta_{1}+\frac{2\varepsilon}{N} & \frac{1}{N}S_{N2}\eta_{N}\eta_{2}+\frac{2\varepsilon}{N} & \frac{1}{N}S_{N3}\eta_{N}\eta_{3}+\frac{2\varepsilon}{N} & \dots & -\hat T_{N}
	\end{pmatrix},\]
\end{linenomath*}
	where
	$\hat T_{i}=\frac{1}{N}\sum_{j=1,j\ne i}^{N}S_{ij}\eta_{i}\eta_{j}+\frac{2\varepsilon (N-1)}{N}$.

   We focus on the stability of the three memories and consider the Jacobians $\mathcal I_{\xi^k}, \, k=1,2,3$.  We   consider  the Jacobian $\mathcal I_{\xi^1}$  only and the other two can be obtained by permutation. 
We introduce the following notation:
 \begin{linenomath*}
	\begin{align*}
	&J_{1}=\left\{j\in [N] \mid \xi^2_j=\xi^3_j \right\},  \quad J_{11}=\left\{j\in J_{1} \mid \xi^2_j=\xi^3_j=\xi^1_j \right\}, \;\;
	J_{12}=\left\{j\in J_{1} \mid \xi^2_j=\xi^3_j=-\xi^1_j \right\},\\
&J_{2}=\left\{j\in [N]   \mid \xi^2_j=-\xi^3_j \right\},\,
 J_{21}=\left\{j\in J_{2} \mid \xi^2_j=-\xi^3_j=\xi^1_j \right\},
	J_{22}=\left\{j\in J_{2} \mid \xi^2_j=-\xi^3_j=-\xi^1_j \right\}.
	\end{align*}\end{linenomath*}
Furthermore, we denote
	\begin{align*}
	  	&J_{11}^{+}=\left\{j\in J_{11}\mid \xi^2_j=\xi^3_j=\xi^1_j=1\right\},  \qquad\,  J_{11}^{-}=\left\{j\in J_{11}\mid \xi^2_j=\xi^3_j=\xi^1_j=-1\right\},\\
&J_{12}^{+}=\left\{j\in J_{12}\mid \xi^2_j=\xi^3_j=-\xi^1_j=1\right\}, \quad \,\, J_{12}^{-}=\left\{j\in J_{12}\mid \xi^2_j=\xi^3_j=-\xi^1_j=-1\right\},\\
	  	&J_{21}^{+}=\left\{j\in J_{21}\mid \xi^2_j=-\xi^3_j=\xi^1_j=1\right\},\quad  \,\, J_{21}^{-}=\left\{j\in J_{21}\mid \xi^2_j=-\xi^3_j=\xi^1_j=-1\right\},\\
	  	&J_{22}^{+}=\left\{j\in J_{22}\mid \xi^2_j=-\xi^3_j=-\xi^1_j=1\right\}, \,\,\,\, J_{22}^{-}=\left\{j\in J_{22}\mid \xi^2_j=-\xi^3_j=-\xi^1_j=-1\right\} .
	  	\end{align*}

 \begin{lemma}\label{eig123lem1}
 	Let $\{\xi^1,\xi^2,\xi^3\}$ be the set of memories in system \eqref{mod3}. For $\mathcal I_{\xi^1}$, we have:\\  
 $(1)$  $0$ is a single eigenvalue with eigenspace $V[\mathbb{1}]$;\\
	$(2)$ $-2\varepsilon$ is a single eigenvalue with eigenspace  $V(J_{1};J_{2})$; \\
 	 $(3)$  if $J_{11}\neq \emptyset$, then $\frac{-3|J_{11}|+|J_{12}|-|J_{2}|}{N}-2\varepsilon$ is an eigenvalue with multiplicity $|J_{11}|-1$,  corresponding to the eigenspace $V(J_{11}^{+};J_{11}^{-})+V(J_{11}^{+})+V(J_{11}^{-})$;\\
 	 $(4)$  if $J_{12}\neq \emptyset$, then $\frac{-3|J_{12}|+|J_{11}|-|J_{2}|}{N}-2\varepsilon$ is an eigenvalue with multiplicity $|J_{12}|-1$,  corresponding to the eigenspace $V(J_{12}^{+};J_{12}^{-})+V(J_{12}^{+})+V(J_{12}^{-})$;\\
  $(5)$ if $J_{11}\neq \emptyset$ and $J_{12} \neq \emptyset$, then $\frac{|J_{1}|-|J_{2}|}{N}-2\varepsilon$ is a single eigenvalue with eigenspace $V(J_{11}; J_{12})$;\\
   $(6)$  if $J_{21}\neq \emptyset$, then $\frac{-3|J_{21}|+|J_{22}|-|J_{1}|}{N}-2\varepsilon$ is an eigenvalue with multiplicity $|J_{21}|-1$, corresponding to the eigenspace $V(J_{21}^{+};J_{21}^{-})+V(J_{21}^{+})+V(J_{21}^{-})$;\\
	 $(7)$  if $J_{22}\neq \emptyset$, then $\frac{-3|J_{22}|+|J_{21}|-|J_{1}|}{N}-2\varepsilon$ is an eigenvalue with multiplicity $|J_{22}|-1$,  corresponding to the eigenspace $V(J_{22}^{+};J_{22}^{-})+V(J_{22}^{+})+V(J_{22}^{-})$;\\
    $(8)$ if $J_{21}\neq \emptyset$ and $J_{22} \neq \emptyset$, then $\frac{|J_{2}|-|J_{1}|}{N}-2\varepsilon$ is a single eigenvalue with eigenspace $V(J_{21}; J_{22})$.
 \end{lemma}
	\begin{proof} Similar to the proof of Lemma \ref{eig_eta}, we only need to show that  any nonzero vector in each space   is an eigenvector of the corresponding  eigenvalue.

$(1)$ The relation $\mathcal I_{\xi^1}\mathbb{1}=0\mathbb{1}$ is obvious, due to  the global phase shift invariance.

 {In order to prove (2)-(8), we first note that} {for any    $x=(x_1,x_2,\dots,x_N)^{\mathrm{T}}$ in those subspaces mentioned in (2)-(8),  we have  $\sum_{\substack{j=1}}^{N}x_j=0$ and
		\begin{align}	\begin{aligned}\label{aaa}
		(\mathcal I_{\xi^1}x^{\star})_i&=\sum_{\substack{j=1,j\neq i}}^{N}\Big(\frac{1}{N}S_{ij}\xi^1_i\xi^1_j+\frac{2\varepsilon}{N}\Big)x_j
+\Big[-\frac{1}{N}\sum_{\substack{j=1,j\neq i}}^{N} S_{ij}\xi^1_i\xi^1_j-\frac{2\varepsilon(N-1)}{N}\Big]x_i\\
		&=\frac{1}{N}\sum_{j=1}^{N}S_{ij}\xi^1_i\xi^1_j(x_j-x_i)-2\varepsilon x_i.
		\end{aligned}\end{align}
Also note that
	\begin{align}\label{equa5} S_{ij}\xi^1_i\xi^1_j=(\xi^1_i\xi^1_j+\xi^2_i\xi^2_j+\xi^3_i\xi^3_j)\xi^1_i\xi^1_j=\begin{cases} 1+\xi^1_i (\xi^2_i+\xi^3_i), &j\in J_{11},\\ 1-\xi^1_i(\xi^2_i+\xi^3_i), &j\in J_{12},\\ 1+\xi^1_i\xi^1_j\xi^2_j(\xi^2_i-\xi^3_i), &j\in J_{2}. \end{cases}
	\end{align}
	Next we prove the statements (2)-(8).}

	$(2)$
	It suffices to show $(\mathcal I_{\xi^1}x^{\star})_i=-2\varepsilon x^{\star}_i$ for $i\in [N]$ where $x^\star=x^\star(J_{1};J_{2})$. 
By \eqref{aaa}, the desired result is equivalent to
	\begin{linenomath*}
		\begin{align}\label{equa1}
		{\sum_{j=1}^{N}S_{ij}\xi^1_i\xi^1_j(x^{\star}_j-x^{\star}_i)\equiv 0, \;\;\text{for} \;\;i\in [N].}
		\end{align}
	\end{linenomath*}
	Using \eqref{equa5}, 
we have
		\begin{align*}
		 \sum_{j=1}^{N}S_{ij}\xi^1_i\xi^1_j(x^{\star}_j-x^{\star}_i)
 &=\Big(\sum_{j\in J_{11}} +\sum_{j\in J_{12}} +\sum_{j\in J_{2}}\Big)S_{ij}\xi^1_i\xi^1_j(x^{\star}_j-x^{\star}_i)\\		
&=\Big[1+\xi^1_i(\xi^2_i+\xi^3_i)\Big]\sum_{j\in J_{11}}(x^{\star}_j-x^{\star}_i)+\Big[1-\xi^1_i(\xi^2_i+\xi^3_i)\Big]\sum_{j\in J_{12}}(x^{\star}_j-x^{\star}_i)\\
		&\quad +\sum_{j\in J_{2}}\Big[1+\xi^1_i\xi^1_j\xi^2_j(\xi^2_i-\xi^3_i)\Big](x^{\star}_j-x^{\star}_i).
		\end{align*}
	If $i\in J_{1}$, then $\sum_{j\in J_{11}}(x^{\star}_j-x^{\star}_i)=0, \sum_{j\in J_{12}}(x^{\star}_j-x^{\star}_i)=0$ and $\xi^2_i-\xi^3_i=0$. 
If $i\in J_{2}$, then $\xi^2_i+\xi^3_i=0$ and $x^{\star}_j-x^{\star}_i=0$ for $j\in J_{2}$. So the equality \eqref{equa1} holds.

	$(3)$
	 It suffices to prove that the nonzero vectors in  $V(J_{11}^{+};J_{11}^{-}),V(J_{11}^{+})$ and $V(J_{11}^{-})$ are eigenvectors corresponding to the eigenvalue $\frac{-3|J_{11}|+|J_{12}|-|J_{2}|}{N}-2\varepsilon$.
	
	For $V(J_{11}^{+};J_{11}^{-})$,
	we need to prove $\mathcal I_{\xi^1}x^{\star}=\big(\frac{-3|J_{11}|+|J_{12}|-|J_{2}|}{N}-2\varepsilon\big) x^{\star} $ where $x^\star=x^\star(J_{11}^{+};J_{11}^{-})$. 
This is equivalent to show
	\begin{linenomath*}
		\begin{align}
		&(\mathcal I_{\xi^1}x^{\star})_i=\big(\frac{-3|J_{11}|+|J_{12}|-|J_{2}|}{N}-2\varepsilon\big) |J_{11}^{-}|,\;\;\text{for} \;\;i\in J_{11}^{+},\label{equa6}\\
		&(\mathcal I_{\xi^1}x^{\star})_i=-\big(\frac{-3|J_{11}|+|J_{12}|-|J_{2}|}{N}-2\varepsilon\big) |J_{11}^{+}|,\;\;\text{for} \;\;i\in J_{11}^{-},\label{equa7}\\
		&(\mathcal I_{\xi^1}x^{\star})_i=0,\;\;\text{for} \;\;i\notin J_{11}^{+}\cup J_{11}^{-}.\label{equa8}
		\end{align}
	\end{linenomath*}
	By \eqref{aaa} and \eqref{equa5}, we obtain
	\begin{linenomath*}
		\begin{align*}
		&\quad	(\mathcal I_{\xi^1}x^{\star})_i=\frac{1}{N}\Big(\sum_{j\in J_{11}^{+}}+\sum_{j\in J_{11}^{-}}+\sum_{j\in J_{12}}+\sum_{j\in J_{2}}\Big)S_{ij}\xi^1_i\xi^1_j (x^{\star}_j-x^{\star}_i) -2\varepsilon x^{\star}_i\\
& {=\frac{1}{N}\left[1+\xi^1_i(\xi^2_i+\xi^3_i)\right]|J_{11}^{+}|(|J^{-}_{11}|
-x^{\star}_i)+\frac{1}{N}[1+\xi^1_i(\xi^2_i+\xi^3_i)]|J_{11}^{-}|(-|J^{+}_{11}|-x^{\star}_i)}\\
		&\quad {-\frac{1}{N}\left[1-\xi^1_i(\xi^2_i+\xi^3_i)\right]|J_{12}|x^{\star}_i
		-\frac{1}{N}\sum_{j\in J_{2}}\left[1+\xi^1_i\xi^1_j\xi^2_j(\xi^2_i-\xi^3_i)\right]x^{\star}_i -2\varepsilon x^{\star}_i.}
		\end{align*}
	\end{linenomath*}
	If $i\in J_{11}^{+}$, then $\xi^1_i(\xi^2_i+\xi^3_i)=2$,  $\xi^2_i-\xi^3_i=0$ and $x^{\star}_i=|J_{11}^{-}|$.  Combining  with $|J_{11}^{+}|+|J_{11}^{-}|=|J_{11}|$,   the equality \eqref{equa6} holds. If $i\in J_{11}^{-}$, then $\xi^1_i(\xi^2_i+\xi^3_i)=2$,  $\xi^2_i-\xi^3_i=0$ and $x^{\star}_i=-|J_{11}^{+}|$.  Combining  with $|J_{11}^{+}|+|J_{11}^{-}|=|J_{11}|$, the equality \eqref{equa7} holds. If $i\notin J_{11}^{+}\cup J_{11}^{-}$, then $x^{\star}_i=0$ and the equality \eqref{equa8} holds.
	
	Secondly, we consider the vectors in $V(J_{11}^{+})$. For any $x=(x_1,x_2,\dots,x_N)^{\mathrm{T}}\in V(J_{11}^{+})$,
	we need to prove
	\begin{linenomath*}
		\begin{align}
		&(\mathcal I_{\xi^1}x)_i=0,\;\;\text{for} \;\;i\notin J_{11}^{+},\label{equa9}\\
		&(\mathcal I_{\xi^1}x)_i=\big(\frac{-3|J_{11}|+|J_{12}|-|J_{2}|}{N}-2\varepsilon\big) x_i,\;\;\text{for} \;\;i\in J_{11}^{+}.\label{equa10}
		\end{align}
	\end{linenomath*}
By \eqref{aaa} and  \eqref{equa5}, we obtain
	\begin{linenomath*}
		\begin{align*}
		(\mathcal I_{\xi^1}x)_i&=\frac{1}{N}\Big(\sum_{j\in J_{11}}+\sum_{j\in J_{12}} +\sum_{j\in J_{2}} \Big)S_{ij}\xi^1_i\xi^1_j (x_j-x_i) -2\varepsilon x_i\\
		&=\frac{1}{N}[1+\xi^1_i(\xi^2_i+\xi^3_i)]\sum_{j\in J_{11}}(x_j-x_i)+\frac{1}{N}[1-\xi^1_i(\xi^2_i+\xi^3_i)]\sum_{j\in J_{12}}(x_j-x_i)\\
		&\quad +\frac{1}{N}\sum_{j\in J_{2}}[1+\xi^1_i\xi^1_j\xi^2_j(\xi^2_i-\xi^3_i)](x_j-x_i)
		 -2\varepsilon x_i.
		\end{align*}
	\end{linenomath*}
	By the definition of $V(J_{11}^{+})$, we see $\sum_{j\in J_{11}^{+}}x_j=0$ and $x_j=0$ for $j\notin J_{11}^{+}$. Then we have
	\begin{align*}
		(\mathcal I_{\xi^1}x)_i&=\frac{1}{N}\big[1+\xi^1_i(\xi^2_i+\xi^3_i)\big](-|J_{11}|x_i)-\frac{1}{N}\big[1-\xi^1_i(\xi^2_i+\xi^3_i)\big]|J_{12}|x_i\\
		&\quad-\frac{1}{N}\sum_{j\in J_{2}}\big[1+\xi^1_i\xi^1_j\xi^2_j(\xi^2_i-\xi^3_i)\big]x_i -2\varepsilon x_i.
	\end{align*}
	If $i\notin J_{11}^{+}$, then $x_i=0$ and the equality \eqref{equa9} holds. If $i\in J_{11}^{+}$, then $\xi^1_i(\xi^2_i+\xi^3_i)=2$ and $\xi^2_i-\xi^3_i=0$; so the equality \eqref{equa10} holds.
	
	The proof for $V(J_{11}^{-})$ is similar, 
so it is omitted here.
	
	$(4)$ The desired result can be derived from that in (3), by replacing  the set of memories  $\{\xi^1,\xi^2,\xi^3\}$ with $\{-\xi^1,\xi^2,\xi^3\}$, which produces the same Hebbian network $\{S_{ij}\}$.  
	
	$(5)$ It suffices to show $\mathcal I_{\xi^1}x^{\star}=\big(\frac{|J_{1}|-|J_{2}|}{N}-2\varepsilon\big) x^{\star}$  where $x^\star=x^\star(J_{11}; J_{12})$. 
	This is equivalent to show
	\begin{linenomath*}
		\begin{align}
		&(\mathcal I_{\xi^1}x^{\star})_i=\big(\frac{|J_{1}|-|J_{2}|}{N}-2\varepsilon\big) |J_{12}|,\quad\text{for} \;\;i\in J_{11},\label{equa2}\\
		&(\mathcal I_{\xi^1}x^{\star})_i=-\big(\frac{|J_{1}|-|J_{2}|}{N}-2\varepsilon\big) |J_{11}|,\quad\text{for} \;\;i\in J_{12},\label{equa3}\\
		&(\mathcal I_{\xi^1}x^{\star})_i=0,\quad\text{for} \;\;i\notin  { J_1=J_{11}\cup J_{12}}.\label{equa4}
		\end{align}
	\end{linenomath*}
	By \eqref{aaa} and   \eqref{equa5}, 
	we obtain
	\begin{linenomath*}
		\begin{align*}
		(\mathcal I_{\xi^1}x^{\star})_i&=\frac{1}{N}\Big(\sum_{j\in J_{11}}+\sum_{j\in J_{12}} +\sum_{j\in J_{2}} \Big)S_{ij}\xi^1_i\xi^1_j (x^{\star}_j-x^{\star}_i)
		 -2\varepsilon x^{\star}_i\\
		&=\frac{1}{N}\left[1+\xi^1_i(\xi^2_i+\xi^3_i)\right]|J_{11}|(|J_{12}|-x^{\star}_i)+
\frac{1}{N}\left[1-\xi^1_i(\xi^2_i+\xi^3_i)\right]|J_{12}|(-|J_{11}|-x^{\star}_i)\\
		&\quad  {-\frac{1}{N}\sum_{j\in J_{2}}\left[1+\xi^1_i\xi^1_j\xi^2_j(\xi^2_i-\xi^3_i)\right]x^{\star}_i
		-2\varepsilon x^{\star}_i.}
		\end{align*}
	\end{linenomath*}
	If $i\in J_{11}$, then $\xi^1_i(\xi^2_i+\xi^3_i)=2$, $\xi^2_i-\xi^3_i=0$ and $x^{\star}_i=|J_{12}|$; so the equality \eqref{equa2} holds. If $i\in J_{12}$, then $\xi^1_i(\xi^2_i+\xi^3_i)=-2$, $\xi^2_i-\xi^3_i=0$ and $x^{\star}_i=-|J_{11}|$; so the equality \eqref{equa3} holds. If $i\notin J_{1}$, then $x^{\star}_i=0$ and $\xi^2_i+\xi^3_i=0$; so the equality \eqref{equa4} holds.

(6)-(8) The desired results can be derived from that in (3)-(5), by replacing  the set of memories  $\{\xi^1,\xi^2,\xi^3\}$ with $\{\xi^1,\xi^2,-\xi^3\}$, which produces the same  $\{S_{ij}\}$.	
\end{proof}

	Now we can present the stability of memories.
	
	\begin{theorem}\label{phi1}
		Let $\{\xi^1,\xi^2, \xi^3\}$ be a set of memories in system \eqref{mod3} satisfying
	\begin{linenomath*}	\[J_{11} \neq \emptyset,\; J_{12} \neq \emptyset,\; J_{21} \neq \emptyset,\; J_{22} \neq \emptyset.\]\end{linenomath*}
	 {Then for $l\in\{1,2,3\}$,
		$\varphi^*(\xi^l)$ is asymptotically stable if $\varepsilon\in (\varepsilon^*_{\xi^l},\infty)$ and unstable if $\varepsilon\in(0,\varepsilon^*_{\xi^l})$, where}
		\begin{align*}\label{epstar123}
		\varepsilon^*_{\xi^1}=\left|\frac{|J_{11}\cup J_{12}|}{N}-\frac{1}{2}\right|, \quad \varepsilon^*_{\xi^2}= \left|\frac{|J_{11}\cup J_{22}|}{N}-\frac{1}{2}\right|\quad \text{and}\quad \varepsilon^*_{\xi^3}= \left|\frac{|J_{11}\cup J_{21}|}{N}-\frac{1}{2}\right|.
		\end{align*}
	\end{theorem}
	\begin{proof} We only show the assertion for $\varepsilon^*_{\xi^1}$, and the others for $\varepsilon^*_{\xi^2}$ or   $\varepsilon^*_{\xi^3}$ can be derived by using the commutation $\xi^2\leftrightarrow \xi^1$ or $\xi^3\leftrightarrow \xi^1$.
		
	 {	Firstly, we note that the eigenvalues in (3) and (4) of Lemma \ref{eig123lem1} are smaller than that in (5), i.e.,
		\[\max\Big\{\frac{-3|J_{11}|+|J_{12}|-|J_2|}{N}-2\varepsilon, \frac{-3|J_{12}|+|J_{11}|-|J_2|}{N}-2\varepsilon \Big\}<\frac{|J_1|-|J_2|}{N}-2\varepsilon.\]
		In fact, we have $|J_{11}|+|J_{12}|=|J_1|$ and
		$\max\Big\{-3|J_{11}|+|J_{12}|,-3|J_{12}|+|J_{11}|\Big\}<|J_1|.$
		Similarly, the eigenvalues in (6) and (7) of Lemma \ref{eig123lem1} are smaller than that in (8). Thus, the maximum eigenvalue of $\mathcal I_{\xi^1}$ is
		\begin{align}\label{maxeig}
		\lambda_{\max}(\mathcal I_{\xi^1})=\max\Big\{\left|\frac{|J_1|-|J_2|}{N} \right|-2\varepsilon,0\Big\}.
		\end{align}}
	
	 {Secondly, we claim that, 	at least one of
		$$\frac{|J_{1}|-|J_{2}|}{N}-2\varepsilon=2\Big[\Big(\frac{|J_{1}|}{N} -\frac{1}{2}\Big)-\varepsilon\Big]  \quad \text{and} \quad
		\frac{|J_{2}|-|J_{1}|}{N}-2\varepsilon=2\Big[\Big(\frac{1}{2}-\frac{|J_{1}|}{N} \Big)-\varepsilon\Big]$$
		appears as an eigenvalue of $\mathcal I_{\xi^1}$. 
		If $\frac{|J_{1}|-|J_{2}|}{N}-2\varepsilon$ is not an eigenvalue of $\mathcal I_{\xi^1}$, then either $J_{11}=\emptyset$ or $J_{12}=\emptyset$.   This means $J_{21}\neq\emptyset$ and $J_{22}\neq\emptyset$, which implies that $\frac{|J_{2}|-|J_{1}|}{N}-2\varepsilon$ is   an eigenvalue of $\mathcal I_{\xi^1}$. In fact, among the index sets $J_{11}, J_{21}, J_{22}$ and $J_{12}$, at least three of them are nonempty. For example, $J_{11}=J_{21}=\emptyset$ implies $J_{12}\cup J_{22}=[N]$ and $\xi^2=-\xi^1$, which contradicts to the fact that $\xi^k\neq \pm\xi^l (k\neq l)$.  }
		
		 {Finally, Using
		\[\varepsilon^*_{\xi^1}=\left|\frac{|J_{11}\cup J_{12}|}{N}-\frac{1}{2}\right|=\left|\frac{|J_{1}|}{N}-\frac{1}{2}\right|=\frac{1}{2}\left|\frac{|J_{1}|-| J_{2}|}{N}\right|,\]
		we can see that when $\varepsilon>\varepsilon^*_{\xi^1}$, the above estimates imply that  all   eigenvalues of $\mathcal I_{\xi^1}$ stated in Lemma \ref{eig123lem1} are strictly less than $0$, 
	except for the eigenvalue $0$ which   is due to the global phase shift invariance. So $\varphi^*(\xi^1)$ is asymptotically stable.  When $\varepsilon<\varepsilon^*_{\xi^1}$, by \eqref{maxeig} we have   $\lambda_{\max}(\mathcal I_{\xi^1})=\left|\frac{|J_1|-|J_2|}{N} \right|-2\varepsilon>0$ and so $\varphi^*(\xi^1)$ is unstable.}
	\end{proof}
	
	\begin{remark}
		{\em In the formulas for critical strengths $\varepsilon^*_{\xi^l}$ in Theorem \ref{phi1}, we have
		\begin{linenomath*}
		\begin{align*}
		J_{11}\cup J_{21}=\left\{j \,|\, \xi^1_j=\xi^2_j \right\},\quad
		  J_{11}\cup J_{22}=\left\{j \,|\, \xi^1_j=\xi^3_j \right\},\quad
		 J_{11}\cup J_{12}=\left\{j  \,|\, \xi^2_j=\xi^3_j \right\}.
		\end{align*}
		\end{linenomath*}
		Theorem \ref{phi1} shows that memories $\xi^1$, $\xi^2$ and $\xi^3$, in general,  are not  $\varepsilon$-independently stable, which differs from the case with two memories (see Theorem \ref{onlytwo}).}
	\end{remark}
	
	Next we present   examples  for  the system \eqref{mod3} with  three memories, in which  it can be impossible to distinguish  the memories from other   patterns by selecting the parameter $\varepsilon$.
		\begin{example}\label{exam1}
		{\em Consider a set of memories $\{\xi^1,\xi^2, \xi^3\}$ in Figure \ref{n789} with $N=968$. By  Theorem \ref{phi1}, we have
		$\varepsilon^*_{\xi^1}
			\approx 0.248, \;
			\varepsilon^*_{\xi^2}
			\approx 0.109,\;
			\varepsilon^*_{\xi^3}
			\approx 0.078.$
	Therefore,   none of $\xi^1,\xi^2$ and $\xi^3$ is $\varepsilon$-independently stable.
	On other hand, we find that the critical strength $\varepsilon^*_{\xi}$ of
		the  pattern $\xi$  in Figure \ref{num789} is approximately $0.173$, by calculating the eigenvalues of Jacobian.  
 It is easy to see that  the pattern $\xi^1$   cannot be distinguished from $\xi$  by selecting  $\varepsilon$, since  $\varepsilon^*_{\xi}<\varepsilon^*_{\xi^1}$. }

		{\em As an extreme case,  let us consider the memories $\{\eta^1,\eta^2, \eta^3\}$ in Figure \ref{n368}   with $N=968$. None  of the memories is $\varepsilon$-independently stable, since
		$
		\varepsilon^*_{\eta^1} 
		\approx 0.215, \;
		\varepsilon^*_{\eta^2}
		\approx 0.170,\;
		\varepsilon^*_{\eta^3}
		\approx 0.160.
		$
		However, the binary pattern $ {\eta}$  in Figure \ref{num368} is $\varepsilon$-independently stable.} 
	
	\end{example}
		\begin{figure*}[htbp]
		\small
		\centering
		\subfigure[$ \xi^1,\xi^2$ and $\xi^3$.]{
			\label{n789}
			\includegraphics[width=4.1cm,height=1.4cm]{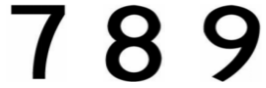}}
		\hspace{0.5cm}
		\subfigure[$\xi$]{
			\label{num789}
			\includegraphics[width=0.8cm,height=1.4cm]{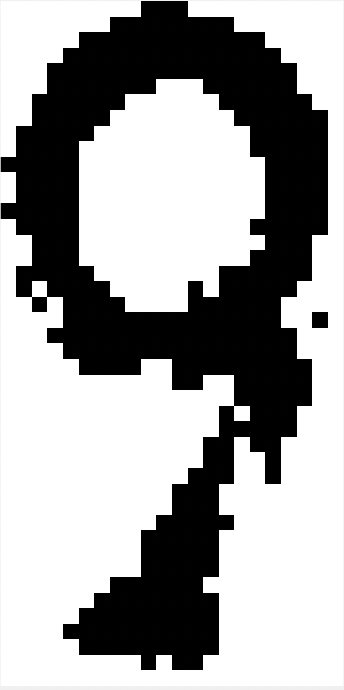}}
		\hspace{1.6cm}
		\subfigure[$\eta^1,\eta^2$ and $\eta^3$.]{
			\label{n368}
			\includegraphics[width=4.4cm,height=1.4cm]{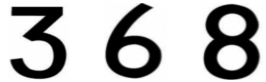}}
		\hspace{0.5cm}
		\subfigure[$ {\eta}$]{
			\label{num368}
			\includegraphics[width=0.8cm,height=1.4cm]{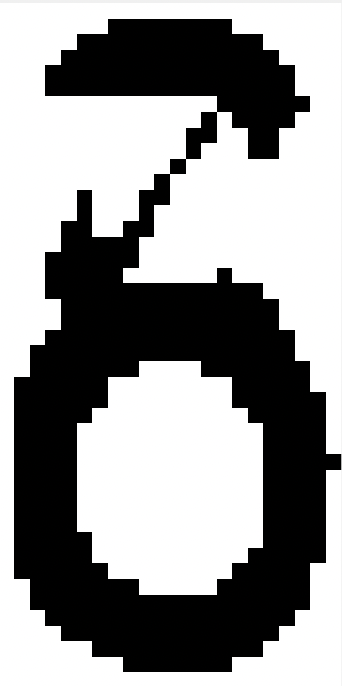}}
		\caption{The sets of memories and binary patterns $\xi, {\eta}$ in Example \ref{exam1}.}
	\end{figure*}

	\section{The approach for general cases and simulations}\label{sec4}
	\setcounter{equation}{0}
	In this section, we propose an approach for the pattern retrieval problems with general standard patterns, inspired by Theorem \ref{onlytwo}. 
	\subsection{Approach for  general standard patterns}
	In this part, we  use an idea of {\em subgrouping}  the standard patterns to reduce the number of memories in each process. Let $\{\eta^1,\eta^2,\dots,\eta^M\}\, (M\ge 2)$ be a set of standard patterns and let $\bar\eta$  be a defective pattern. 
	 The procedure is as follows, referred as	{\em procedure $\mathbf A$}:
	\begin{itemize}
		\item[${\mathbf{(A1)}}$]  Subgroup the standard patterns so that each subgroup contains two patterns. A subgroup containing one pattern is allowed when $M$ is odd.  Then $\lceil \frac{M}{2}\rceil$ subgroups are obtained, where $\lceil \frac{M}{2}\rceil$ denotes the smallest integer not less than $\frac{M}{2}$.
		\item[${\mathbf{(A2)}}$]  For each subgroup, we construct a  pattern retrieval system  memorizing the subgroup and use the dynamical system to retrieve one from the subgroup for the defective pattern $\bar\eta$.  Finally we retrieve $\lceil \frac{M}{2}\rceil$ patterns, denoted by $\eta^{(1,1)}, \eta^{(1,2)}, \dots, \eta^{(1,\lceil \frac{M}{2}\rceil)}$.
		\item[${\mathbf{(A3)}}$]  Use $\big\{\eta^{(1,1)}, \eta^{(1,2)}, \dots, \eta^{(1,\lceil \frac{M}{2}\rceil)}\big\}$ as  the set of standard patterns and perform the same process in Step ${\mathbf{(A1)}}$ and Step ${\mathbf{(A2)}}$.  After finite iterations,   one of the standard patterns is retrieved.
	\end{itemize}
	An algorithm realizing procedure $\mathbf A$ is  given as in Algorithm \ref{algA}.
	
	\begin{algorithm*}[h!]
		\caption{Binary pattern recognition  by subgrouping the standard patterns. }
		\label{algA}
		\begin{algorithmic}
			\State Input standard patterns $\{\eta^1,\eta^2,\dots,\eta^M\}$ and a defective pattern  $\bar\eta$. 
			\State \textbf{Step 1.}  Subgroup the standard patterns so that each one contains one or two  patterns. Output $\lceil \frac{M}{2}\rceil$ subgroups. 
			\State \textbf{Step 2.} For the subgroup containing only one pattern (if exists), output this pattern.   For each subgroup containing two patterns $\{\eta^k,\eta^l\}$, we apply
			procedure ${\mathbf{(A2)}}$.  That is, solve the system of nonlinear equations
			\[{\dot \varphi}_i =\frac{1}{N}\sum_{j=1}^{N}
			(\xi^1_i\xi^1_j+\xi^2_i\xi^2_j)\sin(\varphi_{j} - \varphi_i)+\frac{\varepsilon}{N}\sum_{j=1}^{N}\sin2(\varphi_{j}-\varphi_{i}),\;i=1,2,\dots,N\]
			where $\varepsilon$ is selected according to Theorem \ref{onlytwo} and
		$\xi^1=\eta^k,\; \xi^2=\eta^l,\;  \varphi(0)=\arccos(\bar\eta).$
			If $m(\xi^1)\to 1$, output $\eta^k$; if  $m(\xi^2)\to 1$, output $\eta^l$. Here the function $m$ is defined in \eqref{overlap}.
			\State  \textbf{Step 3.} Collect the outputs  for all subgroups obtained in Step 2, and take the collection as a new set of standard patterns. If the set contains more than one pattern, we redo Step 1-Step 2 with  the new standard pattern set; if the set contains only one pattern, the process stopped and this pattern is the  final output.
		\end{algorithmic}
	\end{algorithm*}
	
	\subsection{Simulation}
	In this subsection, we present the simulation to illustrate the  applications of Algorithm \ref{algA}.
	
	Consider the {mutually nonorthogonal} standard patterns  in Figure \ref{Fignlett} which are  symbols of numbers $\{1,2,\dots,9,0\}$ and letters $\{A,B,\dots,Y,Z\}$.
	The dimension of each pattern is $968$ ($44 \times 22$).
	In the following experiments,
	we consider a binary pattern retrieval problem with a defective copy of  standard patterns.
	
	\begin{figure}[htbp]
		\small
		\centering    
		\subfigure[Number symbols $\{1,2,\dots,9,0\}$.]{
			\label{Fign10}
			\includegraphics[width=9cm, height=1.1cm]{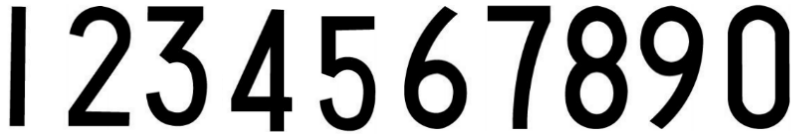}}
		\hspace{0.5cm}
		\subfigure[Letter symbols $\{A,B,\dots,Y,Z\}$.]{
			\label{Figlett26}
			\includegraphics[width=12cm, height=2.1cm]{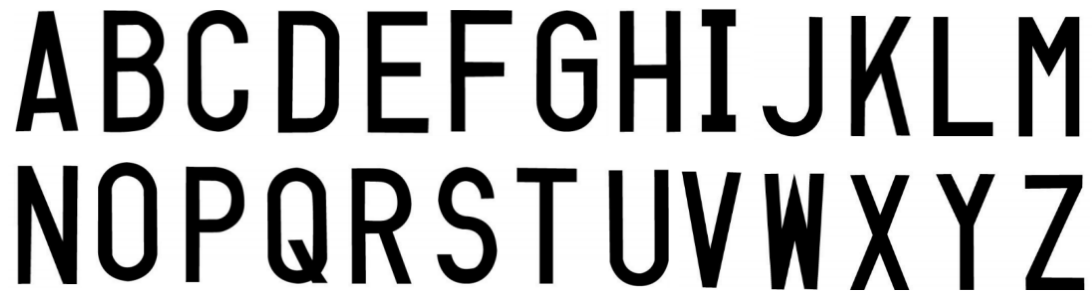}}
		\caption{Standard patterns    with $N=968$.} 
		\label{Fignlett}
	\end{figure}

	As an example, we consider the defective   license plates GB$35024$ with the lower 20 rows sheltered and ZL$76981$ being partially covered by mud,  
shown  in   Figure \ref{lisc}. Simulations show that our approach succeed to retrieve   all the symbols in the license plates.
As an example, the retrieval processes for   the defective symbol  $5$   is described in Figure \ref{hstg5}.
\begin{figure}[htbp] 
		\small
		\centering    
		\subfigure[A defective license plate GB$35024$.]{
			\label{Figcpc}
			\includegraphics[width=7cm, height=1.6cm]{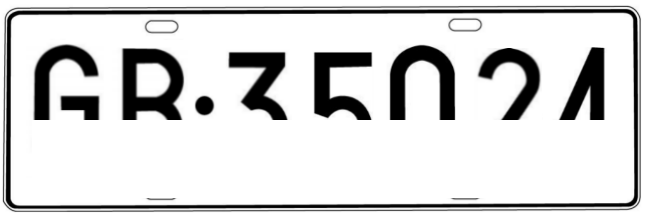}}
		\hspace{0.5cm}
		\subfigure[A defective license plate ZL$76981$.]{
			\label{cpd}
			\includegraphics[width=7cm, height=1.6cm]{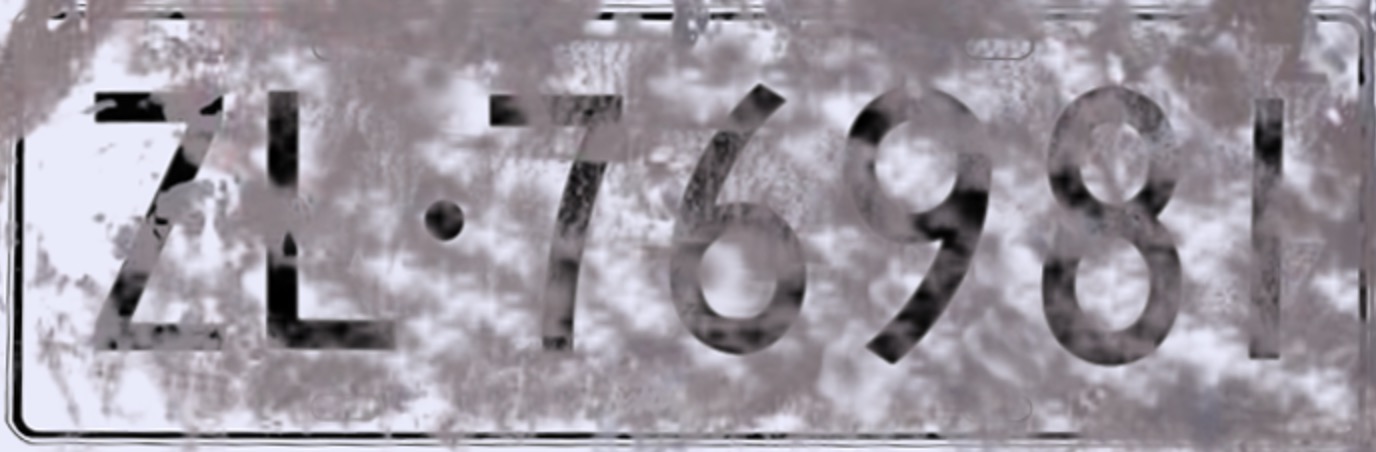}}
		\caption{Two defective license plates.} 
		\label{lisc}
	\end{figure}
	\begin{figure}[htbp]
	\small
	\centering    
		\includegraphics[width=6.5cm, height=6cm]{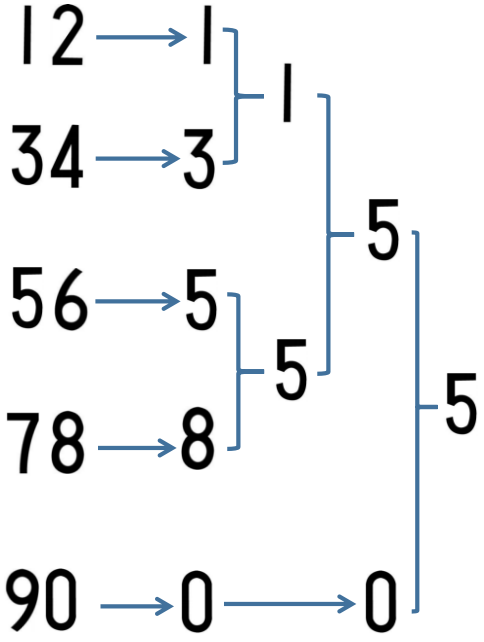}
	\caption{The retrieval processes for the defective symbol $5$ in Figure \ref{Figcpc}.} 	
	\label{hstg5}
\end{figure}

For defective licenses in Figure \ref{Figcpc}-\ref{cpd}, we compared the results of Algorithm \ref{algA} and several public General Optical Character Recognition (General OCR), for example,
Tencent OCR\cite{tencentocr}, Alibaba OCR\cite{alibabaocr} and OCR Space\cite{ocrspace}.  Table \ref{tab1} shows the results of character recognition. 
\begin{table}[htbp]
 \caption{\label{tab1}  Comparison between General OCR  and Algorithm \ref{algA} for Fig. \ref{lisc}.}
   	 \begin{tabular}{|c|c|c|}
    	\hline
    	\textrm{Recognition Platforms}&
    	\textrm{Output for Fig. \ref{Figcpc}}&
    	\textrm{Output for Fig. \ref{cpd}}\\
    	\hline
    	Tencent OCR &  $7099.90$&  $77698$ \\
    	\hline
    	Alibaba OCR &  GRZ50?A&   \text{-Null-} \\
    	\hline
    	OCR Space &   \text{-Null-}&   \text{-Null-}\\
    	\hline
    	Algorithm \ref{algA} & GB$35024$&  ZL$76981$ \\
    	\hline
    	\end{tabular}
    \vspace{0.2cm}
\end{table}

	\section{Conclusion}\label{seccon}
	 In 	this paper, we investigated the Hebbian network of Kuramoto-type oscillators   with general standard patterns. We analyzed the stability/instability of all binary patterns by finding out the spectrum of the Jacobian at each pattern and the  potential analysis for gradient flow, when the system has two memories.  We gave the   critical value for the parameter of the second-order Fourier  term to guarantee the stability of memories and instability of all other patterns.  A  nontrivial estimate for  the basin of   stable patterns was presented.  Based on these results, we designed a  unified approach    for the   application of this model to general pattern retrieval  problems  with small or large errors. Error-free retrieval can be guaranteed which means that only standard patterns can emerge in the dynamical evolution.  The case with  three memories was considered and it was  found that the nice theory  cannot be extended to this case. 
	
	\vspace{0.5cm}
    \textbf{Acknowledgements} Z. Li was supported by the National Natural Science Foundation of China	(Grant  No.   12371509). X. Zhao was supported by the National Natural Science Foundation of China (Grant No.12201156) and the Hong Kong Scholars Scheme (Grant No. XJ2023001). 
	X. Zhou was partially supported by the Research Grants Council General Research Fund of Hong Kong (Grant No. 11308121, 11318522 and 11308323).
	
	\vspace{0.5cm}
	\textbf{Conflicts of Interest}
	The authors declare that there is no conflict of interests regarding thepublication of this paper.
	
	\vspace{0.5cm}
	\textbf{Data availability statement}
	Data sharing not applicable to this article as no datasets were generatedor analyzed during the current study.
	


\begin{thebibliography}{00}
		
		\bibitem {AT}  T. Aoyagi:
		\newblock Network of neural oscillators for retrieving phase information.
		\newblock {\em Phys. Rev. Lett.}, 74 (1995), pp. 4075--4078.
		
		\bibitem {AT2}  T. Aonishi:
		\newblock Phase transitions of an oscillator neural network with a standard Hebb learning rule.
		\newblock {\em Phys. Rev. E}, 58 (1998), pp. 4865--4871.
		
		\bibitem {A-K}  P.-A. Absil and K. Kurdyka:
		\newblock On the stable equilibrium points of gradient systems.
		\newblock {\em Syst. Control Lett.}, 55 (2006), pp. 573--577.
		
		\bibitem {C-H-J-K} Y.-P. Choi, S.-Y. Ha, S. Jung and Y. Kim: \newblock{Asymptotic formation and orbital stability of phase-locked states for the Kuramoto model.} \newblock {\em Physica D},  241 (2012), pp. 735--754.
		
 		\bibitem {C-S} N. Chopra and M. W. Spong: \newblock{On exponential synchronization of Kuramoto oscillators.} \newblock {\em IEEE Trans.
 			Automat. Control},  54 (2009), pp. 353--357.
		
 	\bibitem {D-B} F. D{\"o}rfler and F. Bullo: \newblock{Synchronization and transient stability in power networks and nonuniform Kuramoto oscillators.} \newblock {\em SIAM J. Control Optim.},  50 (2012), pp. 1616--1642.
		
	\bibitem {D-B1} F. D{\"o}rfler and F. Bullo: \newblock{Synchronization in complex networks of phase oscillators: A
survey.} \newblock {\em Automatica},  50 (2014), pp. 1539--1564.


 	\bibitem {D-X}  {J.-G. Dong and X. Xue: \newblock{Synchronization analysis of Kuramoto oscillators.} \newblock {\em Commun. Math. Sci.},  11 (2013), pp. 465--480.}
		
		\bibitem {F-M-R-P}  R. Follmann,  E. E. N. Macau, E. Rosa and J. R. C. Piqueira:
		\newblock Phase oscillatory network and visual pattern recognition.
		\newblock {\em IEEE Trans. Neural Netw. Learn. Syst.}, 26 (2015), pp. 1539--1544.
		
	\bibitem {G-H-M}  T. Girnyk, M. Hasler, and Y. Maistrenko:
		\newblock Multistability of twisted states in non-locally  coupled Kuramoto-type models.
		\newblock {\em Chaos}, 22 (2012), 013114.

		
		\bibitem {H2}  D.-O. Hebb:
		\newblock The Organization of Behavior.
		\newblock {\em Wiley, New York}, (1949).	
		
		\bibitem {H-K}   D. Heger and   K. Krischer:  \newblock Robust autoassociative memory with coupled networks of Kuramoto-type
		oscillators. \newblock {\em Phys. Rev. E},  94 (2016), 022309.
		
		\bibitem {Ho-K1}  R.-W. H{\"o}lzel and  K. Krischer: \newblock Pattern recognition minimizes entropy production in a neural network of electrical oscillators.
		\newblock {\em Phys. Lett. A}, 377 (2013), pp. 2766--2770.
		
		\bibitem {Ho-K}   R.-W. H\"{o}lzel and   K. Krischer:  \newblock Stability and long term behavior of a Hebbian network of Kuramoto oscillators. \newblock {\em SIAM J. Appl. Dyn. Syst.},  14 (2015), pp. 188--201.
		
		\bibitem {H}   J. J. Hopfield: \newblock Neural networks and physical systems with emergent collective computational abilities. \newblock {\em Proc. Natl. Acad. Sci. USA},  79 (1982), pp. 2554--2558.
		
 	\bibitem {H-N-P}  S.-Y. Ha, S. E. Noh, and J. Park:
 		\newblock Synchronization of Kuramoto oscillators with adaptive couplings.
 	\newblock {\em SIAM J. Appl. Dyn. Syst.}, 15 (2016), pp. 162--194.
	



		\bibitem {LL-XX} Z. Li and X. Xue: \newblock{Convergence of analytic gradient-type systems with periodicity and its applications in Kuramoto models.}  \newblock {\em Appl. Math. Lett.}, 90 (2019), pp. 194--201.
		


		\bibitem{L-Z-X} Z. Li, X. Zhao and X. Xue: \newblock Hebbian network of Kuramoto oscillators with second-order couplings for binary pattern retrieve: II. nonorthogonal standard patterns and structural stability. \newblock {\em SIAM J. Appl. Dyn. Syst.}, 21 (2022), pp. 102--136. 
		
		\bibitem{L-F-M}  {Z. Lin, B. Francis and M. Maggiore: \newblock State agreement for continuous-time coupled nonlinear systems. \newblock {\em SIAM J. Control Optim.}, 46 (2007), pp. 288--307.}
		
		\bibitem {N-H-L}  T.  Nishikawa,   F. Hoppensteadt  and  Y.-C. Lai: \newblock Oscillatory associative memory network with perfect retrieval. \newblock {\em Physica D},  197 (2004), pp. 134--148.
		
	\bibitem {P-H} A. N. Pisarchik and A. E. Hramov:
		\newblock What is Multistability. In: Multistability in Physical and Living Systems.
		\newblock {\em Springer, Cham.},  (2022).	



		\bibitem {N-L-H} T.  Nishikawa, Y.-C.  Lai  and  F.  Hoppensteadt: \newblock Capacity of oscillatory associative-memory networks with error-free retrieval. \newblock {\em Phys. Rev. Lett.}, 92 (2004), 108101.
		
		\bibitem {W-P-P} L. Wu, H. R. Pota and I. R. Petersen: \newblock Synchronization conditions for a multirate Kuramoto network with an arbitrary topology and nonidentical oscillators. \newblock {\em IEEE Transactions on Cybernetics}, 49 (2019), 2242-2254.


		\bibitem{Z-L-X} X. Zhao, Z. Li and X. Xue: \newblock Stability in a Hebbian network of Kuramoto
		oscillators with second order couplings for binary pattern retrieve. \newblock {\em SIAM J. Appl. Dyn. Syst.},
		19 (2020), pp. 1124--1159. 
		
		\bibitem{Z-L-X2} X. Zhao, Z. Li and X. Xue: \newblock  Unified approach for applications of oscillatory associative-memory networks with
		error-free retrieval.  \newblock {\em Phys. Rev. E},  108 (2023),  014305.
	

\bibitem{trick}\newblock  \href{https://physics.aps.org/story/v13/st12}{https://physics.aps.org/story/v13/st12}.
	\bibitem{tencentocr}\newblock \href{https://cloud.tencent.com/act/event/ocrdemo}{https://cloud.tencent.com/act/event/ocrdemo/}.
	\bibitem{alibabaocr}\newblock \href{https://duguang.aliyun.com/}{https://duguang.aliyun.com/}.
	\bibitem{ocrspace}\newblock \href{https://ocr.space/}{https://ocr.space/}.
	\end{thebibliography}
\end{document}